\documentclass{amsart}
\usepackage[utf8]{inputenc}

\usepackage{amssymb,amsbsy,amsmath,amsfonts,amssymb,amscd}
\usepackage{mathrsfs}
\usepackage{epsfig, graphicx}
\usepackage{graphics,epsfig}
\usepackage{mathrsfs}
\usepackage{amsthm}
\usepackage{mathtools}
\usepackage[dvipsnames]{xcolor}

\newtheorem{remark}{Remark}

\newtheorem{proposition}{Proposition}

\usepackage{algorithm}
\usepackage{algpseudocode}
\usepackage{caption}
\usepackage{subcaption}

\newcommand{\half}{\frac{1}{2}}
\newcommand{\energy}{\mathcal{E}}

\newcommand{\Amat}{\mathsf{A}}

\newcommand{\Imat}{\mathsf{I}}

\newcommand{\rd}{\mathrm{d}}

\newcommand{\Wass}{\mathcal{W}}

\newcommand{\bx}{{\boldsymbol{x}}}
\newcommand{\bw}{{\boldsymbol{w}}}
\newcommand{\bz}{{\boldsymbol{z}}}
\newcommand{\by}{{\boldsymbol{y}}}
\newcommand{\bv}{\boldsymbol{v}}
\newcommand{\bu}{\boldsymbol{u}}
\newcommand{\bm}{\boldsymbol{m}}
\newcommand{\bT}{\boldsymbol{T}}
\newcommand{\bs}{\boldsymbol{s}}

\newcommand{\RR}{{\mathbb{R}}}
\newcommand{\divergence}{\text{div}}
\newcommand{\diag}{\text{diag}}
\newcommand{\tr}{\text{tr}}
\newcommand{\np}{{n+1}}
\newcommand{\norm}[1]{ \| #1 \|}

\title[Deep JKO for general gradient flows]{
Deep JKO: time-implicit particle methods for general nonlinear gradient flows}
\begin{document}
\author[Lee]{Wonjun Lee}
\email{lee01273@umn.edu}
\address{School of Mathematics, University of Minnesota, MN 55455.}
\author[Wang]{Li Wang}
\email{liwang@umn.edu}
\address{School of Mathematics, University of Minnesota, MN 55455.}
\author[Li]{Wuchen Li}
\email{wuchen@mailbox.sc.edu}
\address{Department of Mathematics, University of South Carolina, Columbia, SC 29208.}

\date{}
\thanks{W. Lee acknowledges support from the National Institute of Standards and Technology (NIST) under award number 70NANB22H021 and partially supported by AFOSR YIP award No. FA9550-23-1-0087. L. Wang acknowledges support from NSF grant DMS-1846854. W. Li's work is supported by AFOSR MURI FP 9550-18-1-502, AFOSR YIP award No. FA9550-23-1-0087, NSF RTG: 2038080, and NSF DMS-2245097. L. Wang and W. Li would like to express their gratitude for the hospitality extended by AIM, where part of the work was discussed. }
\maketitle
\begin{abstract}
    We develop novel neural network-based implicit particle methods to compute high-dimensional Wasserstein-type gradient flows with linear and nonlinear mobility functions. The main idea is to use the Lagrangian formulation in the Jordan--Kinderlehrer--Otto (JKO) framework, where the velocity field is approximated using a neural network. We leverage the formulations from the neural ordinary differential equation (neural ODE) in the context of continuous normalizing flow for efficient density computation. Additionally, we make use of an explicit recurrence relation for computing derivatives, which greatly streamlines the backpropagation process. Our methodology demonstrates versatility in handling a wide range of gradient flows, accommodating various potential functions and nonlinear mobility scenarios. Extensive experiments demonstrate the efficacy of our approach, including an illustrative example from Bayesian inverse problems. This underscores that our scheme provides a viable alternative solver for the Kalman-Wasserstein gradient flow.
\end{abstract}

\section{Introduction}
Deep learning has revolutionized many areas, fundamentally reshaping fields such as natural language processing, visual recognition, and beyond \cite{lecun2015deep}. It has also emerged as the cornerstone of contemporary scientific computing. By harnessing the power of deep neural networks and their ability to approximate complex functions, it has become an invaluable tool for approximating solutions to partial differential equations (PDEs), particularly those of high dimensions. These equations often defy conventional computational techniques due to their complexity and dimensionality.

Several methodologies have employed neural network approximations for solving PDEs. One approach is the deep Ritz method, crafted upon the variational formulation of elliptic-type equations \cite{yu2018deep}. Another avenue is the deep backward stochastic differential equation (BSDE) method, based on the probabilistic and control-based formulation of parabolic equations \cite{han2018solving}. A third technique involves the deep Galerkin method \cite{sirignano2018dgm} and the physics-informed neural network (PINN) \cite{raissi2019physics}. These methodologies both hinge on the minimization of the ($L^2$) residual of the equation. Importantly, they possess the capability to be employed across a wide spectrum of equation types. In parallel, this paper aims to develop a versatile solver for a substantial class of nonlinear PDEs characterized by a gradient flow structure.

Gradient flows  describe the evolution of a function or a probability measure over time, with the rate of change at any point being proportional to the negative gradient of the energy with respect to the function's value. They were originally used to study physical or biological systems that naturally obey some form of least action law. However, they have since emerged as powerful tools in various domains, including sampling, variational inference, the theory of neural networks, and more.


To be more specific, we consider the continuity equation of the form:
\begin{equation}\label{gd}
\partial_t \rho = - \nabla_\bx \cdot (\rho \bv) := \nabla_\bx \cdot [\rho \nabla_\bx (U'(\rho) + V + W \ast \rho )], \quad \rho(0, \cdot) = \rho_0\,.
\end{equation}
Here, $\rho(t,\bx)$, with $\bx\in \mathbb{R}^d$, represents the particle density function, $\rho_0$ is the initial density function, $U(\rho) \in \mathbb{R}^1$ represents the internal potential, $V(\bx) \in \mathbb{R}^1$ is a drift potential, $W(\mathbf{x},\mathbf{y}) = W(\mathbf{y},\mathbf{x}) \in \mathbb{R}^1$ is an interaction potential involving $\mathbf{x}$ and $\mathbf{y}$ in $\mathbb{R}^d$, and $*$ denotes the convolution operator.
This equation can be viewed as the gradient flow of the energy functional 
\begin{equation} \label{eqn:energy}
\energy (\rho) = \int_{\RR^d} \left[U(\rho(\bx)) + V(\bx) \rho(\bx) \right]\rd \bx + \half \int_{\RR^d \times \RR^d} W(\bx-\by) \rho(\bx) \rho(\by) \rd \bx \rd \by 
\end{equation}
in Wasserstein-2 space, which is the space of probability measures equipped with Wasserstein-2 distance, denoted as $\Wass_2$.

Among all the recent endeavors in devising structure-preserving methodologies for \eqref{gd}, our focus in this paper will be exclusively on the Jordan-Kindelenr-Otto (JKO) scheme \cite{JKO}.  This scheme centers on the task of determining $\rho^{k+1} (\cdot) \approx \rho(t^{k+1}, \cdot)$ with $t^{k+1} = (k+1)\Delta t$, given the approximation $\rho^{k+1}(\cdot) \approx \rho(t^{k+1}, \cdot)$:
\begin{align} \label{JKO00}
    \rho^{k+1} \in \arg \min_\rho \{   {\Wass_2(\rho, \rho^k)}^2 +  2 \Delta t \energy(\rho)\}\,.
\end{align}
Here, $ \Wass_2(\rho, \rho^k)$ quantifies the Wasserstein-2 distance between $\rho$ and $\rho^k$. The scheme in \eqref{JKO00} can be conceptualized as a time-implicit variant of the fundamental gradient descent, with the expression on the right-hand side of \eqref{JKO00} being recognized as the Wasserstein proximal operator associated with energy functional $\energy$.

In general terms, there exist two methods for expressing the $\Wass_2$ distance, each leading to distinct implementations of \eqref{JKO00}. The first method is the Eulerian representation, which establishes a connection between two densities through a continuity equation involving a velocity field that minimizes kinetic energy during the transport. In contrast, the second method is grounded in the Lagrangian formulation \cite{CarrilloMatthesWolfram2021_lagrangian}. This approach explores the diffeomorphism that maps one density onto another, a technique that has also gained prominence in modern machine learning applications.

In this paper, we adopt the Lagrangian viewpoint by employing the particle method, coupled with principles drawn from neural ordinary differential equations (neural ODEs). To elaborate, starting from a set of particles $\{\bx_j^0\}$ that are i.i.d. samples from $\rho^0$, we seek a sequence of transformations governing the update of all particles' positions as follows: 
\[
\{\bx_j^0\} \xrightarrow{\bT^1} ~  \{\bx_j^1\}
    ~ \xrightarrow{\bT^2} ~  \{\bx_j^2\}
   \cdots  \,.
\]
Here, $\bx_j^k = \bT^k(\bx_j^{k-1})$. As a consequence, the density evolves in accordance with $\rho^k = \bT^k \sharp \rho^{k-1}$ for $k = 1, 2, \ldots$, where $\sharp$ denotes the pushforward operator. We then approximate each map $\bT^k$ by parameterizing the corresponding vector field using a deep neural network. Thanks to the Jacobi identity, this results in a simple update in estimating the density $\rho^k$; see \eqref{NODE} below.

More precisely, our primary formulation is as follows. To transit from time $t^k$ to $t^{k+1}$, we solve the following minimization problem: 
\begin{subequations}\label{DNJKO}
\begin{equation} \label{neural_JKO}
 \begin{split}
        \theta^*=&\arg\min_{\theta\in \Theta}  ~~\mathbb{E}_{\bx \sim \rho} \Big[  \int_0^1 |\bv_\theta (\tau, \bT(\tau,\bx ))|^2 \rd \tau 
    + 2 \Delta t  \Big(\frac{U(\rho(\bT(1,\bx)))}{\rho(\bT(1,\bx))}+V(\bT(1,\bx))\Big)\Big]\\
    &\hspace{2cm}+\Delta t\mathbb{E}_{(\bx,\by) \sim\rho\otimes \rho}\Big[W(\bT(1,\bx),\bT(1,\by))\Big], 
    \end{split} 
    \end{equation}
such that $\bT(\tau, x)$ satisfies the dynamical constraint:
    \begin{equation}\label{DNT}
        \\\frac{\rd }{\rd \tau} \bT(\tau, \bx) = \bv_\theta  (\tau, \bT(\tau,\bx)), \qquad \bT(0,\bx) = \bx, 
        \end{equation}
and the density $\rho(\bT(1,\bx))$ is computed as first evolving:
        \begin{equation}\label{NODE}
        \frac{\partial}{\partial \tau} \log \det |\nabla_{\bx} \bT(\tau, \bx)|= \divergence (\bv_\theta) (\tau, \bT(\tau,\bx))\,, \quad  \rho(\bT(0,\bx)) = \rho^k(\bx)\,,
    \end{equation}
followed by the pushforward relation:
 \begin{equation} \label{push3}
    \rho(\bT(1,\bx))= \rho^k(\bx)/ \det |\nabla_{\bx} \bT(1,\bx)| \,.  
 \end{equation}
\end{subequations}

In this problem, we employ Benamou-Brenier's dynamic formulation \cite{BB00} to calculate the Wasserstein distance and reformulate the continuity equation using the flow map $\bT(\tau, \bx)$, with $\tau$ serving as an artificial time. The associated velocity field is denoted as $\bv_\theta$, where $\theta\in\Theta$ represents the parameters of a neural network. The notation $\mathbb{E}_{\bx\sim \rho}$ denotes the expectation value for $\bx$ distributed according to the probability density $\rho$, and $\mathbb{E}{(\bx,\by)\sim \rho\otimes\rho}$ represents the expectation over pairs of independent variables $(\bx,\by)$, where $\bx$ and $\by$ are independent and satisfy the joint density $\rho\otimes\rho$. 
Equation \eqref{NODE} is a significant outcome from the perspective of neural ODEs, providing a straightforward means to compute $\det |\nabla_{\bx} \bT(\tau, \bx)|$. Consequently, this facilitates an efficient computation of density as outlined in \eqref{push3}.

It's crucial to highlight that the proposed method goes beyond the current scope of merely approximating Wasserstein gradient flows.
Similar approaches, as outlined in \eqref{DNJKO}, can be employed for general gradient flows, including nonlinear mobility Wasserstein and Kalman-Wasserstein gradient flows. In such cases, the identity \eqref{NODE}, utilized in neural ODE density estimations, takes on even greater significance. This is due to the fact that the objective function in \eqref{DNJKO} becomes dependent on the density, introducing additional nonlinearities into the optimization objective (loss function), see \eqref{Mob-JKO} in the subsequent context. 


The exploration of computing \eqref{gd} using the JKO scheme has experienced rapid growth. Previous endeavors employed finite difference or finite volume schemes \cite{BCL16, LiWang22, LiLuWang20, wang2022hessian, CarrilloWangWei2023_structurea, FuOsherLi, cances2020variational}, as well as Lagrangian methods \cite{CarrilloMatthesWolfram2021_lagrangian}. More recently, more efforts have been directed towards developing machine learning-based methods, particularly in tackling high-dimensional challenges. In this context, neural networks have been utilized to approximate maps or density functions \cite{NEJKO,fan2021variational,mokrov2021large, HuLiuWangXu2022_energetic,  shen2022self, boffi2023probability}. Furthermore, a variety of methods have arisen to approximate different formulations of one-step JKO schemes. These range from leveraging generative adversarial networks \cite{WPGAN} to employing neural ODEs \cite{RuthottoOsherLiNurbekyanFung2020_machinee}. These approaches primarily target the sampling of high-dimensional target distributions in machine learning applications \cite{VidalWuFungTenorioOsherNurbekyan2023_taminga,XuChengXie2022_invertiblea}, manifesting as time-implicit updates of the projected dynamics within neural network spaces, specifically the natural gradient dynamics \cite{WPGAN}.

Contrasting previous methodologies, our approach employs neural ODE techniques to compute the time-implicit update of generalized Wasserstein-type gradient flows. This innovative strategy facilitates the streamlined computation of high-dimensional densities in Lagrangian coordinates.  It is worth highlighting the similarities and disparities between our approach and its closest counterparts. One such sibling is the blob method \cite{blob}. While both approaches rely on particles and intrinsically exhibit energy dissipation, the dissipation in the blob method hinges on the chosen time step, whereas ours remains unconditionally stable.
Furthermore, our method eliminates the need for kernel density estimation, a significant computational bottleneck present in the blob method. 
Another approach involves the direct use of a neural network to approximate the mapping between $\rho^k$ and $\rho$ in \eqref{JKO00} \cite{mokrov2021large, HuLiuWangXu2022_energetic}. This approach can be seen as a special case of our method when the inner time step $\tau$ is set to 1. A critical distinction, however, lies in our utilization of the density evolution equation \eqref{NODE} from the perspective of neural ODEs and the explicit formula \eqref{grad-formula} for computing derivatives, which significantly enhances the efficiency of density estimation. Furthermore, our approach can be readily extended to accommodate the nonlinear mobility case.
A third comparable technique is the self-consistent velocity matching approach \cite{boffi2023probability,shen2022self}. In this method, neural networks are also employed to approximate the velocity field. However, instead of sequentially learning it within the JKO formulation as we do, this approach adopts an end-to-end learning strategy. While this method might lead to faster computation, it doesn't inherently preserve energy decay in general.

This paper is organized as follows. In section \ref{sec2}, we review the general formulation of Wasserstein gradient flows in Lagrangian coordinates and introduce the Lagrangian Wasserstein proximal operator and its particle version for time-implicit (i.e., JKO) updates of gradient flows. We expand our framework to include gradient flows characterized by metrics reliant on general nonlinear mobility functions.  Section \ref{sec3} outlines our proposition to approximate the pushforward map using the trajectory of neural ODEs. This approach empowers us to employ neural networks with diverse depths while simultaneously sidestepping the computationally intensive task of back-propagation.
To illustrate the effectiveness and versatility of our approach, we present a series of numerical examples in Section \ref{sec4}, including Fokker-Planck equations, porous media equations, Kalman-Wasserstein gradient flow of Kullback–Leibler (KL)  divergence, and nonlinear mobility gradient flows. Finally, the paper concludes in Section \ref{sec5} with an overview of our findings and insights into future research directions.

\section{Lagrangian JKO formulation}\label{sec2}
In this section, we present the main formulation of this paper. We first review the Lagrangian coordinates of Wasserstein gradient flows and then introduce our variational time-implicit schemes. One major component of our formulation is to use instantaneous change of variable formula (see \eqref{push2}) in the continuous normalizing flow to achieve efficient density computation. We also consider the generalized nonlinear gradient flows, in which the Wasserstein-type metric is associated with nonlinear mobility functions.

\subsection{Dynamic JKO schemes in Lagrangian coordinates}
Recall the dynamic JKO formulation to \eqref{gd} as follows. 
Given $\rho^n(\bx)$, then one can obtain $\rho^{n+1}(\bx)$ as $\rho(1,\bx)$ with $\rho(t,\bx)$ solving 
\begin{equation} \label{classicalJKO}
\begin{dcases}
& (\rho,\bv)=\arg\inf_{ \rho,\bv}~ \int_0^1\int_{\RR^d} \rho |\bv|^2 \rd \bx \rd \tau  + 2\Delta t\energy(\rho(1,\cdot))
\\ & \textrm{s.t.} \quad  \partial_\tau \rho + \nabla \cdot (\rho \bv) = 0, ~ \rho(0,\bx) = \rho^n (\bx)\,.
\end{dcases}
\end{equation}
In the above formulation, the minimization is over the probability density function $\rho(t,\bx)$, $t\in(0,1)$, terminal time density $\rho(1,\bx)$, and vector fields $v(t,\bx)$, such that the continuity equation holds.  

Now we translate the variational problem \eqref{classicalJKO} into the Lagrangian formulation. Assume the velocity field is sufficiently regular, then the solution $\rho(\tau,\bx)$ to the constrained continuity equation can be written as 
\begin{align} \label{push0}
    \rho(\tau,\cdot) = \bT(\tau, \cdot) \sharp \rho^n\,, 
\end{align}
where $\bT$ is the flow map that solves the following ODE:
\begin{equation}\label{ODE}
    \frac{\rd }{\rd \tau} \bT(\tau, \bx) = \bv (\tau, \bT(\tau,\bx)), \qquad \bT(0,\bx) = \bx \,.
\end{equation}
Note that \eqref{push0} implies that for any integrable test function $\phi$, we have that
\begin{align*}
    \int_{\RR^d} \phi(\bx) \rho(\tau, \bx) \rd \bx 
    = \int_{\RR^d} \phi(\bT(\tau,\bx)) \rho^n (\bx) \rd \bx\,.
\end{align*}
Then variational problem \eqref{classicalJKO} rewrites as 
\begin{equation} \label{JKO-Lag}
    \begin{dcases}
    \min_{\bv} \int_0^1 \int_{\RR^d} \rho^n(\bx ) |\bv(\tau, \bT(t,\bx))|^2 \rd \tau \rd \bx +  2 \Delta t \energy(\bT(1, \cdot) \sharp \rho^n)
    \\ \text{s.t.} ~~\frac{\rd}{\rd \tau} \bT(\tau, \bx) = \bv (\tau, \bT(\tau,\bx)), \qquad \bT(0,\bx) = \bx \,.
    \end{dcases}
\end{equation}

To proceed, it is necessary to derive an explicit formulation of $\energy(\bT(1, \cdot) \sharp \rho^n)$ that can be interpreted as an expectation of a functional, allowing for its representation using particles in the subsequent discussion. In the following, we present a comprehensive formulation for each component of equation \eqref{eqn:energy}. Specifically, the external potential and interaction potential have straightforward expectation formulations through the push forward relations, whereas the internal energy necessitates further reformulation.

\subsubsection*{External potential energy}
\begin{align} \label{eng-ext}
    \int_{\RR^d}  V(\bx) \rho(1,\bx) \rd \bx = \int_{\RR^d} V(\bT(1,\bx)) \rho^n(\bx) \rd \bx\,.
\end{align}

\subsubsection*{Interaction energy}
\begin{align} \label{eng-interact}
    & \int_{\RR^d \times {\RR^d}}  W(\bx,\by) \rho(1,\bx) \rho(1,\by) \rd \bx \rd \by \nonumber 
   \\ & \qquad  = \int_{\RR^d \times {\RR^d}} W(\bT(1,\bx), \bT(1, \by)) \rho^n(\bx) \rho^n(\by)\rd \bx \rd \by\,.
\end{align}

\subsubsection*{Internal energy/entropy} 
\begin{align} \label{eng-internal}
     & \int_{\RR^d} U(\rho(1,\bx)) \rd \bx  \nonumber
    \\ & = \int_{\RR^d} U(\rho(1, \bT(1,\bx))) \det |\nabla_{\bx} \bT(1,\bx)|  \rd \bx \nonumber
    \\ &  = \int_{\RR^d} U (\rho(1, \bT(1,\bx))) 
    \frac{\det |\nabla_{\bx} \bT(1,\bx)|}{\rho^n(\bx)} \rho^n(\bx)  \rd \bx \nonumber
    \\ & = \int_{\RR^d} U (\rho(1, \bT(1,\bx))) 
   \frac{1}{\rho(1,\bT(1,\bx))} \rho^n(\bx)  \rd \bx \,,
\end{align}
where we have used the Monge-Ampere equation
\begin{equation} \label{push1}
   \rho(1,\bT(1,\bx)) \det |\nabla_{\bx} \bT(1,\bx)| = \rho^n(\bx) 
\end{equation}
thanks to \eqref{push0}. In practice, computing the determinant of the Jacobian $\nabla_{\bx} \bT(1,\bx))$ poses a significant bottleneck due to its cubic cost with respect to the dimension of $\bx$. To address this issue, we draw inspiration from the concept of continuous normalizing flow and employ an efficient approach based on the instantaneous change of variable formula \cite{chen2018neural, grathwohl2018ffjord}. The crucial element is the following formula.
\begin{proposition}
\begin{align} \label{push2}
\frac{\rd}{\rd \tau} \log \det |\nabla_{\bx} \bT(\tau, \bx)|= \divergence (\bv) (\tau, \bT(\tau,\bx)).
\end{align}
\end{proposition}
\begin{proof}
Equation \eqref{push2} is derived through a calculus that involves both ODE \eqref{ODE} and the Monge-Ampere equation \eqref{push1}. Assume that $\bT(\tau, \bx)$ is invertible, i.e. $\mathrm{det}|\nabla_\bx \bT(\tau,\bx)|\neq 0$, then
\begin{equation*}
\begin{split}
\frac{\rd}{\rd \tau} \log \det |\nabla_{\bx} \bT(\tau, \bx)|=&\mathrm{tr}\Big((\nabla_\bx \bT(\tau, \bx))^{-1}\frac{\rd}{\rd \tau}\nabla_\bx \bT(\tau,\bx)\Big)\\
 =&  \mathrm{tr}\Big((\nabla_\bx \bT(\tau, \bx))^{-1}\nabla_x\frac{\rd }{\rd \tau} \bT(\tau,\bx)\Big)\\
  =&  \mathrm{tr}\Big((\nabla_\bx \bT(\tau, \bx))^{-1}\nabla_x \bv(\tau, \bT(\tau,\bx))\Big)\\
  =& \divergence (\bv) (\tau, \bT(\tau,\bx)),
\end{split}
\end{equation*}
where the first equality uses the Jacobi identity, and the last equality uses the chain rule. 
\end{proof}

As a result, the computational burden shifts towards calculating $\divergence (\bv) (\tau, \bT(\tau,\bx))$, which only involves a trace calculation and can be accomplished much more efficiently. In sum, \eqref{push1} and \eqref{push2} play a crucial role in enabling efficient density estimation, as elaborated in the next section and Section~\ref{sec:learn0}.

\subsection{A Particle method}
By interpreting the integral against $\rho(\bx) \rd \bx$ as an expectation, \eqref{JKO-Lag} reveals a direct particle representation. More precisely, let $\{ \bx_j^n\}_{j=1}^{N}$ be particles sampled from $\rho^n(\bx)$, we discretize \eqref{JKO-Lag} as follows. Note here we omit the superscript $n$ in $\bx_j^n$. 
\begin{itemize}
\item[Case 1]: \eqref{JKO-Lag} with \eqref{eng-ext} and/or \eqref{eng-interact}
\begin{equation} \label{JKO-NN1}
    \begin{dcases}
    \min_{\bv}  \frac{1}{N}\sum_{j=1}^N \left[  \int_0^1 |\bv (\tau, \bT(\tau,\bx_j ))|^2 \rd \tau  + 2 \Delta t V(\bT(1, \bx_j)) + 2 \Delta t \sum_{l=1} ^N W(\bT(1,\bx_j), \bT(1, \bx_l))\right]  
    \\ \text{s.t.} ~~\frac{\rd}{\rd \tau} \bT(\tau, \bx_j) = \bv  (\tau, \bT(\tau,\bx_j)), \qquad \bT(0,\bx_j) = \bx_j \,.
    \end{dcases}
\end{equation}

\item[Case 2]: \eqref{JKO-Lag} with \eqref{eng-internal}
\begin{equation} \label{JKO-NN2}
    \begin{dcases}
    \min_{\bv}  \frac{1}{N}\sum_{j=1}^N \left[  \int_0^1 |\bv (\tau, \bT(\tau,\bx_j ))|^2 \rd \tau  + 2 \Delta t  U\left( \frac{\rho^n(\bx_j)}{\det |\nabla_{\bx} \bT(1,\bx_j)|}\right)\frac{\det |\nabla_{\bx} \bT(1,\bx_j)|}{\rho^n(\bx_j)} \right]  
    \\ \text{s.t.} ~~\frac{\rd }{\rd \tau} \bT(\tau, \bx_j) = \bv  (\tau, \bT(\tau,\bx_j)), \qquad \bT(0,\bx_j) = \bx_j 
    \\ \qquad    \frac{\partial}{\partial \tau} \log \det |\nabla_{\bx} \bT(\tau, \bx_j)|= \divergence (\bv) (\tau, \bT(\tau,\bx_j))\,, \quad \log \det |\nabla_{\bx} \bT(0, \bx_j)|=0\,,
    \end{dcases}
\end{equation}
\end{itemize}
In cases where the energy solely comprises potential and interaction components, it suffices to monitor the particle positions. However, when internal energy is introduced, it becomes necessary also to track the density $\rho$, which is accomplished by monitoring the logarithmic determinant of the transport map.

More precisely, denote the minimizer of either problem $\bT^n$, then 
\[
\bT^n(1, \bx_j^n) =: \bx_j^\np\,,
\]
which can be viewed as samples from $\rho^\np(\bx)$. Therefore, starting from $\{\bx_j^0\}$, we have a sequence of update
\begin{align}\label{xtime}
     \{\bx_j^0\} \xrightarrow{\bT^1} ~  \{\bx_j^1\}
    ~ \xrightarrow{\bT^2} ~  \{\bx_j^2\}
    ~ \xrightarrow{\bT^3} ~  \{\bx_j^3\} \cdots  \,,
\end{align}
where $\bx_j^k = \bT^k(1, \bx_j^{k-1})$. Likewise, the density evolves as 
\begin{align}\label{Ttime}
    \rho^0 \xrightarrow{\bT^1} ~ \rho^1
    ~ \xrightarrow{\bT^2} ~ \rho^2 
    ~ \xrightarrow{\bT^3} ~ \rho^3 \cdots\,,
\end{align}
where $\rho^k = \bT^k(1,\cdot) \sharp  \rho^{k-1}$.

In practice, since we only evolve particles, a major difficulty in \eqref{Ttime} lies in the {\it density estimation} of $\rho^n$, which can be very expensive and inaccurate in high dimensions. To resolve this issue, a key observation is that we don't need the full information of the density, but only the density evaluated along the trajectory of particles! 

Suppose we are given a set of samples ${\mathbf{x}_j^0}$ drawn from the known analytical expression of $\rho_0$. In the first JKO step, we can compute the updated density as follows:
\begin{align} \label{density1}
 \rho^1(\underbrace{\bT^1(1,\bx_j^0)}_{\bx_j^1}) = \frac{\rho^0(\bx_j^0)}{\det|\nabla_\bx \bT^1(1,\bx_j^0)|}\,,   
\end{align}
where $\bT^1$ represents the map obtained in the first JKO step. Here, $\bx_j^1$ corresponds to the transformed sample after applying $\bT^1$ to $\bx_j^0$. Following a similar procedure, we obtain the density in the subsequent JKO steps:
\begin{align} \label{density2}
 \rho^2(\underbrace{\bT^2(1,\bx_j^1)}_{\bx_j^2}) = \frac{\rho^1(\bx_j^1)}{\det|\nabla_\bx \bT^2(1,\bx_j^1)|}\,,
\end{align}
and this iterative process can be continued for any subsequent JKO step. 

\begin{remark}
We emphasize that this observation holds true for general nonlinear internal potential functions, denoted as $U$. In the existing literature, $U$ is commonly assumed to be the negative Boltzmann-Shannon entropy, given by $U(z) = z\log z$. In this particular case, the aforementioned difficulty can be circumvented. Indeed, since 
\begin{equation*}
   U\left( \frac{\rho^n(\bx_j)}{\det |\nabla_{\bx} \bT(1,\bx_j)|}\right)\frac{\det |\nabla_{\bx} \bT(1,\bx_j)|}{\rho^n(\bx_j)}= \log\rho^n(\bx_j)-\log \mathrm{det}|\nabla_x \bT(1,\bx_j)|,
\end{equation*}
we can simplify the computation by avoiding the need to directly calculate $\rho^n(\bx_j)$. This is because $\log\rho^n(\bx_j)$ is separated from  $\log \det (\nabla_{\bx} \bT(1, \bx_j))$ in equation \eqref{JKO-NN2}, enabling us to exclude it from the calculation. However, for general potential functions, the computation of $\rho^n(\bx_j)$ cannot be circumvented.
\end{remark}

\subsection{Nonlinear mobility}
In this section, we extend the previous method to the generalized gradient flow by considering the following equation:
\begin{equation} \label{nonM}
   \partial_t \rho = \nabla_\bx  \cdot \left(M(\rho) \nabla_\bx \frac{ \delta \energy(\rho)}{\delta \rho} \right)\,,
\end{equation}
Here, the function $M(\rho) \geq 0$ represents a nonlinear mobility function. This type of equation appears in various contexts, such as thin films \cite{bertozzi1998mathematics},  phase separation in binary alloys described by the Cahn-Hilliard equation \cite{cahn1961spinodal}, and the transport phenomena of biological systems with overcrowding prevention \cite{burger2006keller}, among others.

For general mobility function $M(\rho)$, equation \eqref{nonM} can also be viewed as a  gradient flow. The generalized metric also admits a dynamic formulation, leading to the following extension of the dynamic JKO scheme \eqref{classicalJKO}: 
\begin{equation} \label{nonM_JKO}
\begin{dcases}
& (\rho,\tilde \bv)=\arg\inf_{ \rho, \tilde \bv}~ \int_0^1\int_{\RR^d} M(\rho) |\tilde \bv|^2 \rd \bx \rd \tau  + 2\Delta t\energy(\rho(1,\cdot))
\\ & \textrm{s.t.} \quad  \partial_\tau \rho + \nabla \cdot (M(\rho) \tilde \bv) = 0, ~ \rho(0,\bx) = \rho^n (\bx)\,.
\end{dcases}
\end{equation}
In the above formulation, the minimization is over density function $\rho(t,\bx)$, $t\in(0,1)$ terminal time density $\rho(1,\bx)$, and vector fields $\tilde v(t,\bx)$, such that the nonlinear mobility induced continuity equation holds.  

Let $\bv = \frac{M(\rho)}{\rho} \tilde \bv$, equation \eqref{nonM_JKO} can be rewritten into 
\begin{equation} \label{nonM_JKO2}
\begin{dcases}
& (\rho, \bv)=\arg\inf_{ \rho, \tilde \bv}~ \int_0^1\int_{\RR^d} \frac{\rho^2}{M(\rho)} | \bv|^2 \rd \bx \rd \tau  + 2\Delta t\energy(\rho(1,\cdot))
\\ & \textrm{s.t.} \quad  \partial_\tau \rho + \nabla \cdot (\rho \bv) = 0, ~ \rho(0,\bx) = \rho^n (\bx)\,.
\end{dcases}
\end{equation}
This reformulation states that the minimization is constrained by the classical continuity equation. It is worth noting that such reformulation does not change the optimizer of the problem. Indeed, let $\phi$ be the Lagrangian multiplier, then one can check that the critical point system of variational problems \eqref{nonM_JKO} and \eqref{nonM_JKO2} lead to the same Hamilton-Jacobi equation for $\phi$:
\begin{align*}
    \partial_t \phi + \frac{1}{4} M'(\rho) |\nabla \phi|^2 = 0\,.
\end{align*}

We note that variational problem $\eqref{nonM_JKO2}$ now admits a similar Lagrangian representation as in $\eqref{JKO-Lag}$, with the only distinction being the first term, which now becomes:
\begin{align*}
    \int_0^1  \int_{\RR^d} \frac{\rho^2(\tau,\bx)}{M(\rho(\tau,\bx))}|\bv(\tau,\bx)|^2  \rd \tau \rd \bx \,.
\end{align*}
By employing the same technique as utilized in deriving \eqref{eng-internal}, the integral with respect to $\mathbf{x}$ in the given expression rewrites as
\begin{align} \label{non1}
     & \int_{\RR^d} \frac{\rho^2(\tau,\bx)}{M(\rho(\tau,\bx))}|\bv(\tau,\bx)|^2    \rd \bx  \nonumber
    \\  =& \int_{\RR^d} \frac{\rho^2(\tau,\bT(\tau,\bx))}{M(\rho(\tau,\bT(\tau,\bx)))} |\bv(\tau,\bT(\tau,\bx))|^2 \det |\nabla_{\bx} \bT (\tau, \bx)|  \rd \bx \nonumber
    \\   =& \int_{\RR^d} \frac{\rho^2(\tau,\bT(\tau,\bx))}{M(\rho(\tau,\bT(\tau,\bx)))} 
    |\bv(\tau,\bT(\tau,\bx))|^2
    \frac{\det |\nabla_{\bx} \bT(\tau, \bx)|}{\rho^n(\bx)} \rho^n(\bx)  \rd \bx \nonumber
    \\ = &\int_{\RR^d} \frac{\rho^2(\tau,\bT(\tau,\bx))}{M(\rho(\tau,\bT(\tau,\bx)))} 
   \frac{|\bv(\tau,\bT(\tau,\bx))|^2}{\rho(\tau,\bT(\tau,\bx))} \rho^n(\bx)  \rd \bx \nonumber
    \\ = &\int_{\RR^d} \frac{\rho(\tau,\bT(\tau,\bx))}{M(\rho(\tau,\bT(\tau,\bx)))} 
    |\bv(\tau,\bT(\tau,\bx))|^2 \rho^n(\bx)  \rd \bx\,.
\end{align}
As in \eqref{density1}, the density $\rho(\tau,\mathbf{T}(\tau,\mathbf{x}))$ is determined by the Monge-Amper{\'e} equation:
\begin{align*} 
 \rho(\tau, \bT(\tau,\bx)) = \frac{\rho^n(\bx)}{\det|\nabla_\bx \bT(\tau,\bx)|}.   
\end{align*}
Similar to the JKO scheme presented in \eqref{JKO-NN1} or \eqref{JKO-NN2}, given $\{\bx_j^n\}_j$ and $\{\rho^n(\bx^n_j)\}_j$ (we omit the superscript $n$ in $\bx_j^n$ in the following formula), the updated formulation can now be expressed as follows:
\begin{equation} \label{Mob-JKO}
    \begin{dcases}
    \min_{\bT}  \frac{1}{N}\sum_{j=1}^N \left[  \int_0^1  \frac{\rho( \bT(\tau,\bx_j)) }{M (\rho(\bT(\tau,\bx_j))) }|\bv (\tau, \bT(\tau,\bx_j ))|^2 \rd \tau  + 2 \Delta t  \frac{U(\rho(\bT(1,\bx_j)))}{\rho(\bT(1,\bx_j))}\right]  
    \\ \text{s.t.} ~~\frac{\rd}{\rd \tau} \bT(\tau, \bx_j) = \bv  (\tau, \bT(\tau,\bx_j)), \qquad \bT(0,\bx_j) = \bx_j 
    \\  \qquad  \frac{\partial}{\partial \tau} \log \det |\nabla_{\bx} \bT(\tau, \bx_j)|= \divergence (\bv) (\tau, \bT(\tau,\bx_j))\,, \quad \log \det |\nabla_{\bx} \bT(0, \bx_j)|=0
    \\ \qquad  \rho(\tau, \bT(\tau,\bx_j)) = \frac{\rho^n(\bx_j)}{\det|\nabla_\bx \bT(\tau,\bx_j)|}. 
    \end{dcases}
\end{equation}
It is obvious that when $M(\rho) = \rho$, \eqref{Mob-JKO} reduces to \eqref{JKO-NN2}.

\section{Neural ODE empowered JKO schemes}\label{sec3}
In this section, we leverage neural network functions to approximate the vector field $\mathbf{v}$ as described in equations \eqref{JKO-NN1}, \eqref{JKO-NN2}, or, more broadly, equation \eqref{Mob-JKO}. As a result, the optimization procedure aims at refining the parameters
of the neural network. Building upon the efficiency gains achieved through the use
of continuous normalizing flow for density estimation, as previously discussed, we incorporate a recursive relation formula proposed in a prior work \cite{onken2021ot} to compute derivatives, significantly improving the efficiency of backpropagation. Lastly, we conclude this section by presenting a concise summary of the algorithm.

\subsection{Learning the potential function}\label{sec:learn0}
In view of \eqref{JKO-NN1} and \eqref{JKO-NN2}, the primary objective is to determine the optimal transport map, denoted as $\bT$. However, this task can be computationally demanding, especially in high-dimensional scenarios. Fortunately, by applying the principles of optimal transport \cite{gangbo1996geometry, villani2009optimal}, we can express the velocity field $\mathbf{v}(t,\bx)$ as the gradient of a potential function $\phi(t,\bx)$ such that $\mathbf{v} (t,\bx)= \nabla_{\mathbf{x}} \phi (t,\bx)$. Consequently, our focus shifts to finding the potential function $\phi$. To achieve this, we utilize a neural network to approximate $\phi(t,\bx)$, denoted as $\phi_\theta(t,\bx)$. As a result, we obtain $\mathbf{v}_\theta(t,\bx) = \nabla_{\mathbf{x}} \phi_\theta(t,\bx)$ as the corresponding approximation of the velocity field.

To be more specific, let us consider solving equation \eqref{JKO-NN1}. By substituting $\bv$ with $\bv_\theta = \nabla_\bx \phi_\theta$, the constraint in \eqref{JKO-NN1}  becomes: 
\[
\frac{\rd}{\rd \tau} \bT(\tau, \bx_j) = \nabla_\bx \phi_\theta  (\tau, \bT(\tau,\bx_j)), \qquad \bT(0,\bx_j) = \bx_j. 
\]

Similarly, in the case of the more complex equation \eqref{JKO-NN2}, computing $\rho(\bT(1,\bx_j))$ calls for the following update:
\begin{align*}\label{0319}
    \frac{\partial}{\partial \tau} \log \left| \det(\nabla_\bx^2 \phi_\theta(\tau, \bT(\tau,\bx_j))) \right|
    = \divergence (\bv_\theta) (\tau, \bT(\tau,\bx_j)) = - \tr(\nabla_\bx^2 \phi_\theta(\tau, \bT(\tau,\bx_j))).
\end{align*}
Detailed computations of $\nabla^2\phi_\theta$ will be given in the next subsection.

\subsection{Deep JKO Algorithms}
We hereby provide comprehensive details for computing the minimizer of~\eqref{JKO-NN1}, \eqref{JKO-NN2} or \eqref{Mob-JKO}, when $\bv$ is replaced by $\bv_\theta$. Let $N_\tau$ denote the number of time discretization such that $[0,1]$ is discretized into $0, \frac{1}{N_\tau}, \frac{2}{N_\tau}, \cdots, , \frac{N_\tau-1}{N_\tau}$, and $\theta$ is a vector of parameters for neural network functions. Given the density $\rho^n$ at the $n$-th JKO step and sampled points $\{\mathbf{x}_j\}^N_{j=1}$, the loss function to compute the density at the $n+1$-th JKO step is defined as follows:
\begin{equation}\label{eq:discrete-NN}
\begin{aligned}
L(\theta, \{\mathbf{x}_j\}^N_{j=1}) = \frac{1}{N}\sum_{j=1}^N &\bigg[ d\tau \sum_{k=0}^{N_\tau-1} \left|\nabla_\bx \phi_\theta(k d\tau,\mathbf{z}_j^k) \right|^2\\
&+ 2 \Delta t \left(\frac{U(\rho(\mathbf{z}_j^{N_\tau}))}{\rho(\mathbf{z}_j^{N_\tau})}  + V(\mathbf{z}^{N_\tau}_j)  + \frac{1}{N}\sum^N_{l=1}  W(\mathbf{z}^{N_\tau}_j,\mathbf{z}^{N_\tau}_l)  \right)\bigg],    
\end{aligned}
\end{equation}
where the inner time step is denoted as $d\tau = 1/N_\tau$ and the update equations are given by:
\begin{align*}
\mathbf{z}_j^{k+1} &= \mathbf{z}_j^k - d\tau\, \nabla_\bx \phi_\theta(k d\tau, \mathbf{z}_j^k), \quad \mathbf{z}_j^0 = \mathbf{x}_j, \\
l^{k+1}_j &= l^k_j - d\tau \,\tr ( \nabla_\bx^2 \phi_\theta(k d\tau, \mathbf{z}_j^k )), \quad l^0_j = \log \left|\det\left(\nabla_\bx^2 \phi_\theta(k d\tau, \mathbf{z}_j^k ) \right)\right|,
\end{align*}
for $k=0,\cdots, N_\tau-1$ and the terminal density $\rho$  at the location $\bz^{N_\tau}_j$ is defined  as
\begin{equation}\label{eq:update-density}
    \rho(\bz^{N_\tau}_j) = \frac{\rho^n(\mathbf{x}_j)}{\exp(l^{N_\tau}_j)}. 
\end{equation}
The above equations for $\mathbf{z}_j^{k+1}$ and $\log\rho(\mathbf{z}_j^{k+1})$ utilize the forward Euler method to solve the ordinary differential equations (ODEs) in the constraints of~\eqref{JKO-NN2}. However, alternative ODE solvers such as Runge-Kutta can also be employed. For our numerical experiment, we utilize the RK4 as the ODE solver.

The loss function, denoted as $L(\theta,  \{\bx_j\}^N_{j=1})$, takes two inputs: the neural network parameter $\theta$ and a set of sampled points $\{\bx_j\}^N_{j=1}$. The objective is to minimize this loss function with respect to $\theta$. This $\theta$ plays a vital role in approximating the potential function $\phi_\theta(\bx, \tau)$, which relies on two variables: a state vector $\bx \in \mathbb{R}^d$ and a time value $t\in[0,1]$. The velocity is computed as:
\[
    v_\theta(\tau,  \bx) = - \nabla_\bx \phi_\theta(\tau, \bx),
\]
and its divergence is obtained as the trace of the Hessian:
\[
    \divergence(v_\theta)(\tau, \bx) = - \tr(\nabla_\bx^2 \phi_\theta(\tau, \bx)).
\]

To simplify the calculation of gradients and Hessians in the formulas above, we utilize the explicit expression introduced in \cite{onken2021ot}. This approach capitalizes on the smoothness of the activation function within the neural network. Consequently, our optimization process experiences significant acceleration, as we no longer rely on automatic differentiation (AD) for backpropagation to compute gradients and Hessians in each iteration.

More precisely, we adopt a multi-layer ResNet that takes the input $\mathbf{s} = (\tau, \mathbf{x}) \in \mathbb{R}^{d+1}$, where $\mathbf{x} \in \mathbb{R}^d$ represents the initial state, and $\tau \in [0,1]$ denotes time, yielding an output in $\mathbb{R}$. The neural network function consists of $L$ hidden layers with recursive relations defined as follows:
\begin{align*}
\mathbf{u}_1 &= \sigma (W_0 \mathbf{s} + \mathbf{b}_0),\\
\mathbf{u}_{l+1} &= \mathbf{u}_{l} + \sigma (W_{l} \mathbf{u}_{l} + \mathbf{b}_{l}), \quad l=1,\cdots L-1,\\
NN(\mathbf{s};\theta) &= \bw^\top \mathbf{u}_{L}.
\end{align*}
Here, $\theta=(W_0, W_l,b_l)$, with $W_{0} \in \mathbb{R}^{m \times (d+1)}$ and $W_{l} \in \mathbb{R}^{m\times m}$ $(l=1,\cdots,L)$ are dense matrices. Additionally, $b_k\in\mathbb{R}^m$ $(l=0,\cdots, L)$ are biases, while $m$ denotes the number of nodes in each hidden layer and $\bw \in \mathbb{R}^m$ is a vector we perform dot product at the end of the ResNet. The element-wise activation function $\sigma(\mathbf{x}) = \log(\exp(\mathbf{x}) + \exp(-\mathbf{x}))$ is used, which is differentiable. The dual variable $\phi_\theta$ is defined through
\begin{align} \label{phi}
    \phi_\theta(\tau, \bx) = NN((\tau,\bx);\theta).
\end{align}
Denote by $\nabla_\bx$ a gradient operator with respect to the space variable $\bx$ and $\nabla_\mathbf{s}$ be a gradient operator with respect to the state vector $\mathbf{s}$. Both the gradient $\nabla_\bx \phi_\theta(\bx, \tau)$ and the Hessian $\nabla_\bx^2 \phi_\theta(\bx, \tau)$ can be effortlessly computed using the explicit formulas presented in~\cite{onken2021ot}. Specifically, for the case where $L=2$, these formulas are as follows:
\begin{subequations}\label{grad-formula}
\begin{equation}
    \nabla_{\mathbf{s}} \phi_\theta(\mathbf{x},\tau) = W_0^\top \diag(\sigma'(W_0 \bs + b_0)) \bz_1,
\end{equation}
\begin{equation}
    \bz_1 =  \bw + W_1^\top \diag(\sigma'(W_1\bu_1 + b_1))\bw,
\end{equation}
and
\begin{equation}
  \nabla_{\bs} ^2 \phi_\theta(\bx,\tau) = t_0 + t_1,   
\end{equation}
where
\begin{align*}
    t_0 &= W_0^\top \diag(\sigma'' (W_0 \bs + b_0) \odot \bz_1 ) W_0,\\
    t_1 &= J_0^\top W_1^\top \diag(\sigma'' (W_1 \bu_1 + b_1) \odot \bw) W_1 J_0,\\
    J_0 &= W_0^\top \diag(\sigma'(W_0 \bs + b_0)).
\end{align*}
\end{subequations}

Here $\odot$ represents the element-wise multiplication. Then, $\nabla_\bx \phi_\theta(\bx,\tau)$ is the first $d$ elements of $\nabla_\mathbf{s} \phi_\theta(\bx,\tau)$, and $\nabla_\bx^2\phi_\theta(\bx,\tau) = [\nabla^2_\mathbf{s} \phi_\theta(\bx,\tau)]_{1:d,1:d}$ where the subscript $1:d,1:d$ denotes the submatrix of size $d\times d$. For comprehensive expressions of the gradient and Hessian operators for a general $L$, we recommend readers to see details in \cite{onken2021ot}.

We now provide a summary of the proposed algorithms. Algorithm 2 serves as the primary algorithm, guiding the movement of particles along the gradient flow. Whenever sampling/resampling is necessary, Algorithm 1 will be invoked. Please note that the $l$ update and density computation in Algorithm 1 are only required when working with general nonlinear internal energy and mobility. Additionally, in the algorithm, we utilize forward Euler, but it is adaptable for replacement with explicit Runge-Kutta type ODE solvers.

\begin{algorithm}
\caption{Generate samples and densities at $n$-th JKO step}\label{alg:sampling}
\begin{algorithmic}
\State \textbf{Input:} 
\begin{itemize}
    \item Outer JKO time step $\Delta t$.
    \item Inner dynamic formulation time step $d\tau$ and total number of inner steps $N_\tau$.
    \item Initial distribution $\rho^0$ and initial samples $.\{\bx_j^0\}_{j=1}^N$ and densities $\{\rho(\bx_j^0)\}_{j=1}^N$
    \item Potential functions from previous steps $(\phi_\theta^0, \cdots, \phi_\theta^{n-1} )$.
\end{itemize}

\State \textbf{Output:} 
\begin{itemize}
    \item Sampled points $ \{ \mathbf{x}^{n}_j \}^N_{j=1}$ at $n$-th outer iteration.
    \item Density values $ \{\rho(\bx_j^n)\}^N_{j=1}$ at $n$-th outer iteration.
\end{itemize}

\bigskip

\For{$i=0,\cdots,n-1$}
    \For{$j=1,\cdots,N$}
        \State $\mathbf{z}^0_j = \mathbf{x}^i_j$, $l^0_j = \log \left|\det\left(\nabla_\bx^2 \phi_\theta(k d\tau, \mathbf{z}_j^k ) \right)\right|$ 
        \For{$k=0,\cdots,N_\tau-1$}
\State $\mathbf{z}_j^{k+1} = \mathbf{z}_j^k - d\tau\, \nabla_\bx \phi_\theta(k d\tau, \mathbf{z}_j^k)$
\State
$l^{k+1}_j = l^k_j - d\tau \,\tr ( \nabla_\bx^2 \phi_\theta(k d\tau, \mathbf{z}_j^k ))$
        \EndFor
        \State $\mathbf{x}^{i+1}_j = \mathbf{z}^{N_\tau}_j$, $ \rho(\bx^{i+1}_j) = \frac{\rho^n(\mathbf{x}_j)}{\exp(l^{N_\tau}_j)}$
    \EndFor
\EndFor
\end{algorithmic}
\end{algorithm}

\begin{algorithm}[ht!]
\caption{Deep JKO scheme}\label{alg:cap}
\begin{algorithmic}
\State \textbf{Input:}
\begin{itemize}
\item Outer JKO time step $\Delta t$
\item Inner dynamic formulation time step $d\tau$
\item Initial distribution $\rho^0$
\item Learning rate $\alpha$ for the Adam optimizer
\item Total number of outer iterations $K$
\item Error tolerance $TOL$ for the optimization process
\item Resampling frequency $C$
\end{itemize}

\State \textbf{Output:} 
\begin{itemize}
    \item Neural network parameter $\theta$ and corresponding potential function $\phi_\theta^n(t,\bx)$, $n= 1,2 , \cdots, K$
    \item Particle locations: $\{\bx_j^n\}$, $n= 1,2 , \cdots, K$, $j = 1, 2, \cdots, N$.
    \item Density along particle trajectories: $\rho^n (\bx_j^n)$, $n= 1,2 , \cdots, K$, $j = 1, 2, \cdots, N$.
\end{itemize}

\bigskip

\State Initialize $\theta$ from a standard normal distribution.

\bigskip

\For{$n=1,\cdots,K$}
    \State $c = 0$.
    \While{error $> TOL$}
        \If {$mod(c,C)$ = 0} \textbf{ do}
            \State Sample $\{\bx_j^0\}^N_{j=1} \overset{\text{i.i.d.}}{\sim} \rho^0$
            \State Compute $\{\bx_j^n\}^N_{j=1}$ using  Algorithm~\ref{alg:sampling}
            and  update the density value through  \eqref{eq:update-density}.
        \EndIf
        \State Update $\theta \gets \theta - \alpha \nabla_{\theta} L(\theta, \{\bx_j^n\}^N_{j=1})$
        \State where $L$ is from~\eqref{eq:discrete-NN} and the gradient is  computed using automatic 
        \State differentiation (AD).
        \State $c \gets c + 1$
    \EndWhile
\EndFor
\end{algorithmic}
\end{algorithm}
In Algorithm~\ref{alg:cap}, the error in the stopping condition is defined as
\begin{align*}
    \text{error} = \frac{|L(\theta^{(current)},\{\bx_j^n\}^N_{j=1}) - L(\theta^{(previous)},\{\bx_j^n\}^N_{j=1})|}{|L(\theta^{(previous)},\{\bx_j^n\}^N_{j=1})|}
\end{align*}
where $\theta^{(previous)}$ and $\theta^{(current)}$ are the parameters computed at the previous and current iterations.

\section{Numerical examples}\label{sec4}
This section presents several numerical examples of the proposed deep JKO schemes for computing gradient flows with various choices of energy functionals and mobility functions. 


\subsection{Wasserstein gradient flow of the KL divergence}

In this experiment, we present the computed solutions of~\eqref{JKO-NN2} with the energy functional defined as
\begin{align} \label{KL}
    U(\rho) = \mathrm{D}_{\mathrm{KL}}(\rho\|q)=\int \rho(x)\log\frac{\rho(x)}{q(x)}dx,
\end{align}
where $q$ represents a reference density.
The initial density $\rho^0$ is characterized by two Gaussian distributions centered at  $(\pm 1.2, 0)$ with a standard deviation of $\sigma = 0.5$. The reference density $q$ is defined as four Gaussian distributions centered at $(\pm 2, \pm 2)$ with a standard deviation of $\sigma = 0.5$.

For the numerical experiment, the inner timestep is set as $N_\tau=1$ and the JKO time stepsize is set as $\Delta t=0.025$. The model architecture includes 3 layers with $m=64$ nodes in the hidden layers. The training is performed using a batch size of $1000$, and the learning rate is chosen to be $10^{-5}$. The total number of iterations is set to $10,000$.

The two-dimensional results are depicted in Figure~\ref{fig:rho-log-rho} and Figure~\ref{fig:rho-log-rho-3d}. Figure~\ref{fig:rho-log-rho} illustrates the evolution of the density from $t=0$ to $t=0.4$. In this figure, the color of the points represents the value of the density $\rho^{n}(x)$ at the corresponding location $x$. Brighter colors indicate higher density values. The computation involves a total of 16 JKO steps. Figure~\ref{fig:rho-log-rho-3d} displays the trajectories of each particle with respect to time $t$ and the energy value computed from the algorithm with respect to the iterations. As the iterations progress, the energy decreases and eventually converges to the stationary value.

We further extended our experiments to a $10$-dimensional space while retaining the same energy functional \eqref{KL}. The initial density $\rho^0$ was set as a Gaussian distribution centered at the origin, and the reference density $q$ was constructed using a combination of four Gaussian distributions centered at $(2,2.5)$, $(-2.5,2)$, $(-2,-2.5)$, and $(2.5,-2)$. In this experimental setup, we utilized a time step size of $\tau=0.2$, a learning rate of $10^{-5}$, and executed a total of 10,000 iterations. The results are illustrated in Figure~\ref{fig:rho-log-rho-10d}, highlighting the initial 2D cross sections of the density's evolution from $t=0$ to $t=3.4$. These visual representations present kernel density plots derived from the point clouds generated by the algorithm at each time instance $t$. 

\begin{figure}[th!]
     \centering
     \begin{subfigure}[b]{0.19\textwidth}
         \centering
         \includegraphics[width=\textwidth]{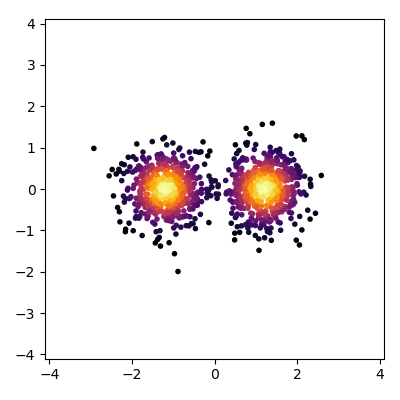}
         \caption{$t=0$}
     \end{subfigure}
     \hfill
     \begin{subfigure}[b]{0.19\textwidth}
         \centering
         \includegraphics[width=\textwidth]{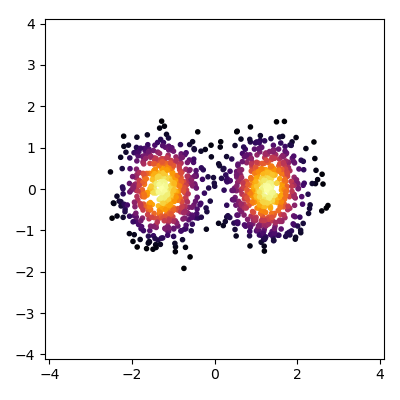}
         \caption{$t=0.025$}
     \end{subfigure}
     \hfill
     \begin{subfigure}[b]{0.19\textwidth}
         \centering
         \includegraphics[width=\textwidth]{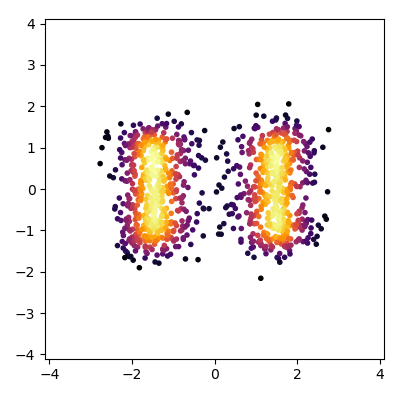}
         \caption{$t=0.1$}
     \end{subfigure}
     \begin{subfigure}[b]{0.19\textwidth}
         \centering
         \includegraphics[width=\textwidth]{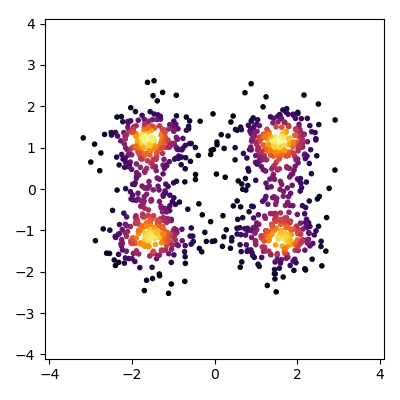}
         \caption{$t=0.2$}
     \end{subfigure}
     \begin{subfigure}[b]{0.19\textwidth}
         \centering
         \includegraphics[width=\textwidth]{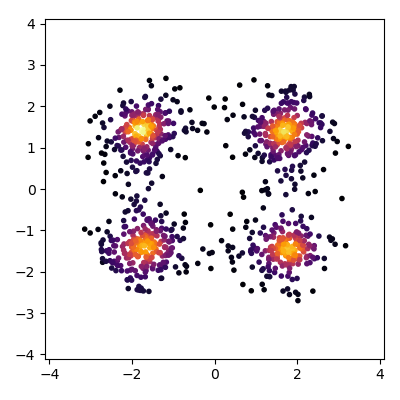}
         \caption{$t=0.4$}
     \end{subfigure}

        \caption{Computed solutions of Fokker-Planck equation, i.e. the Wasserstein gradient flow of KL divergence \eqref{KL}. An initial density is a mixture of two Gaussians and the target density is a mixture of four Gaussians.  The figures visually depict the evolution of these densities from $t=0$ to $t=0.4$.}
        \label{fig:rho-log-rho}
\end{figure}

\begin{figure}[th!]
     \centering
     \begin{subfigure}[b]{0.45\textwidth}
         \centering
         \includegraphics[width=\textwidth]{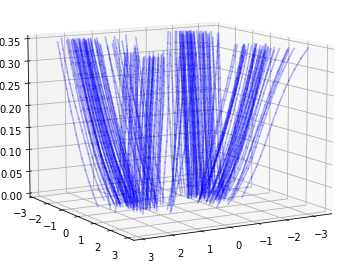}
     \end{subfigure}
     \begin{subfigure}[b]{0.50\textwidth}
         \centering
         \includegraphics[width=\textwidth]{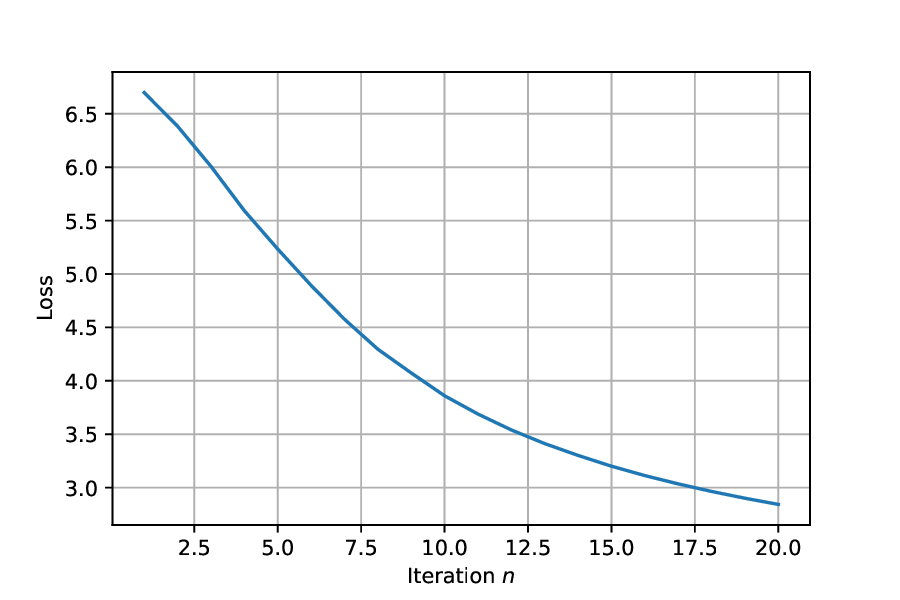}
     \end{subfigure}

        \caption{3D plot (left) and the energy decay plot (right) from Figure~\ref{fig:rho-log-rho}. The left plot shows the trajectories of each particles where $z$-axis represents time from $t=0$ to $t=0.35$. The right plot illustrates the decay of energy over 20 outer iterations (i.e., JKO steps) from $t=0$ to $t=0.5$.}
        \label{fig:rho-log-rho-3d}
\end{figure}

\begin{figure}[th!]
     \centering
     \begin{subfigure}[b]{0.19\textwidth}
         \centering
         \includegraphics[width=\textwidth]{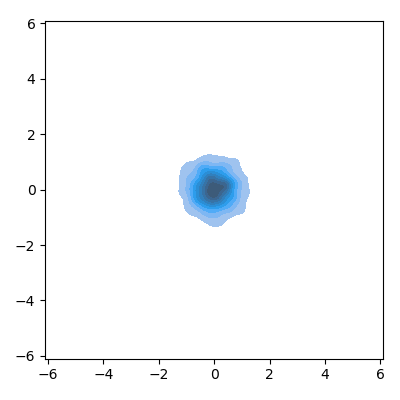}
         \caption{$t=0$}
     \end{subfigure}
     \hfill
     \begin{subfigure}[b]{0.19\textwidth}
         \centering
         \includegraphics[width=\textwidth]{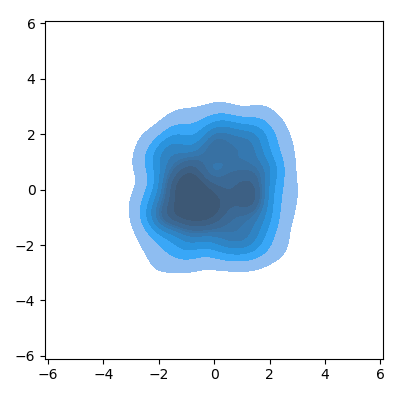}
         \caption{$t=0.6$}
     \end{subfigure}
     \hfill
     \begin{subfigure}[b]{0.19\textwidth}
         \centering
         \includegraphics[width=\textwidth]{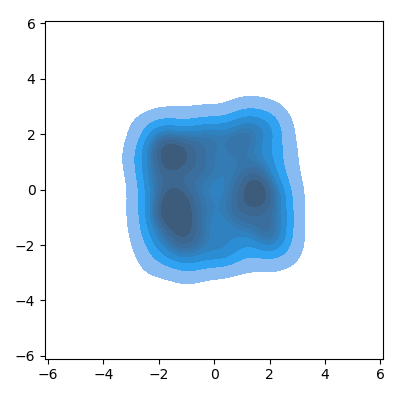}
         \caption{$t=1$}
     \end{subfigure}
     \begin{subfigure}[b]{0.19\textwidth}
         \centering
         \includegraphics[width=\textwidth]{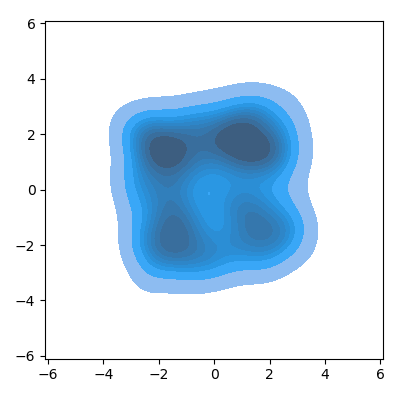}
         \caption{$t=1$}
     \end{subfigure}
     \begin{subfigure}[b]{0.19\textwidth}
         \centering
         \includegraphics[width=\textwidth]{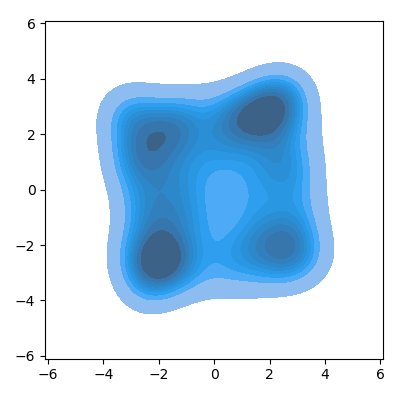}
         \caption{$t=3.4$}
     \end{subfigure}

        \caption{Evolution of densities to  the  Fokker-Planck equation in $\mathbb{R}^{10}$ from $t=0$ to $t=3.4$. An initial density is a Gaussian distribution centered at the origin and the target density is a mixture of four Gaussians.}
        \label{fig:rho-log-rho-10d}
\end{figure}


\subsection{2D porous medium equation}
The second experiment concerns the porous medium equation represented by
\begin{equation}\label{eq:porous-pde}
    \partial_t \rho(t,\bx) = \Delta \rho(t,\bx)^m, \quad m >1\,, \bx \in \RR^d\,,
\end{equation}
which can be interpreted as the Wasserstein gradient flow with the internal energy functional $U$ defined as
\begin{align*}
    U(\rho) = \int_{\mathbb{R}^d} \frac{1}{m-1}\rho(\bx)^m d\bx.
\end{align*}
This equation admits a close-form solution given by the Barenblatt formula: 
\begin{align}\label{BB}
    \rho(t,\bx) = (t+t_0)^{-\alpha} \left( C - \beta \|\bx\|^2 (t+t_0)^{-\frac{2\alpha}{d}}\right)_{+}^{\frac{1}{m-1}} ,
\end{align}
where $\alpha = \frac{d}{d(m-1)+2}$, $\beta = \frac{(m-1)\alpha}{2dm}$, $C$ is a constant that makes $\int_{\RR^d} \rho(t,\bx) \rd x = 1$. We choose $t_0 = 10^{-3}$. Formula \eqref{BB} serves as a reference solution to evaluate the performance of our algorithm.

In the numerical experiment, we employ an inner timestep of $N_\tau=2$ and set the JKO time step size to $\Delta t=0.001$. The model architecture consists of 3 layers with $128$ nodes in each hidden layer. During training, a batch size of 1000 is used, and the learning rate is set at $10^{-5}$. The total number of iterations for the experiment is configured to be  $10,000$.

The experimental results are presented in Figure~\ref{fig:porous1} and Figure~\ref{fig:porous-3d}. Figure~\ref{fig:porous1} visualizes the density evolution from $t=0$ to $t=0.12$, where particles positions $\bx$ and their corresponding density values, $\rho(t, \bx)$, are accurately computed at each time step. The color intensity at each particle's location represents its density, with brighter colors indicating higher density values. Notably, the computed solutions maintain radial symmetry, providing a precise representation. The green rings in the figures correspond to the support of the analytical solution \eqref{BB} and align well with the computed solutions, affirming the algorithm's accuracy. Figure~\ref{fig:porous-3d} displays particle trajectories from $t=0$ to $t=0.01$, exhibiting linear trajectories as expected from the analytical solution of the PDE.

We then utilize the algorithm to compute solutions for the porous medium equation in high-dimensional spaces, specifically in dimensions 6 and 10. For both cases, we choose an outer time step size of $\Delta t=0.0005$, and set the inner timestep to $N_\tau=2$. The model architecture comprises three layers with $128$ nodes in the hidden layers. Training is carried out using a batch size of 1000, and the learning rate is fixed at $10^{-6}$. The total number of iterations is set to be $20,000$.

Figure~\ref{fig:porous10d} illustrates the evolution of densities in both 6D and 10D. To visualize the results and demonstrate the accuracy of the algorithm, 2D cross-section plots of the particles are presented, along with the support of the densities obtained from the analytical solution, represented as a green ring. The computed densities exhibit spherical evolution and closely match the analytical solutions.



\begin{figure}[th!]
    \centering
    \begin{subfigure}[b]{0.19\textwidth}
        \centering
        \includegraphics[width=\textwidth]{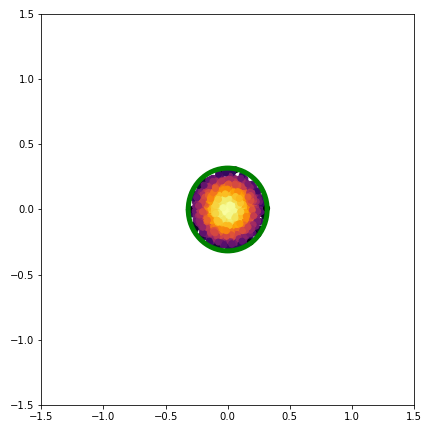}
        \caption{$t=0$}
    \end{subfigure}
    \hfill
    \begin{subfigure}[b]{0.19\textwidth}
        \centering
        \includegraphics[width=\textwidth]{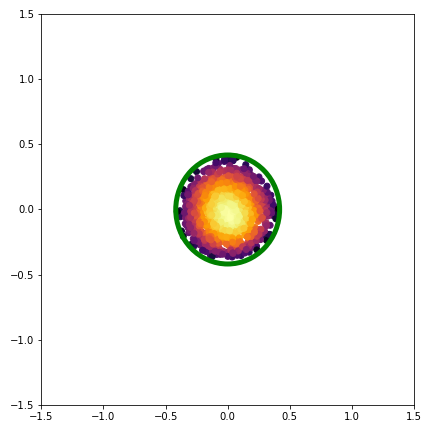}
        \caption{$t=0.002$}
    \end{subfigure}
    \hfill
    \begin{subfigure}[b]{0.19\textwidth}
        \centering
        \includegraphics[width=\textwidth]{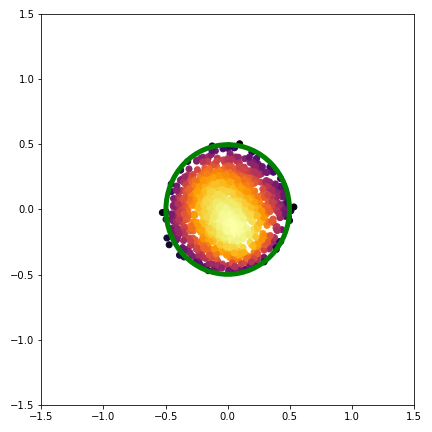}
        \caption{$t=0.005$}
    \end{subfigure}
    \begin{subfigure}[b]{0.19\textwidth}
        \centering
        \includegraphics[width=\textwidth]{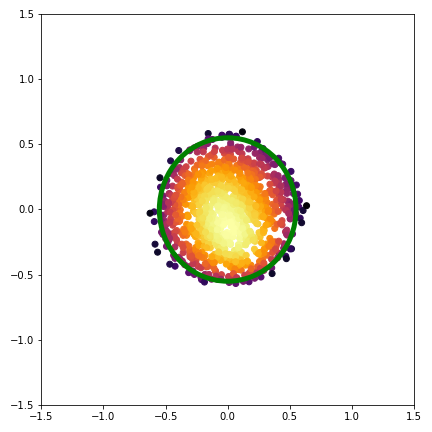}
        \caption{$t=0.008$}
    \end{subfigure}
    \begin{subfigure}[b]{0.19\textwidth}
        \centering
        \includegraphics[width=\textwidth]{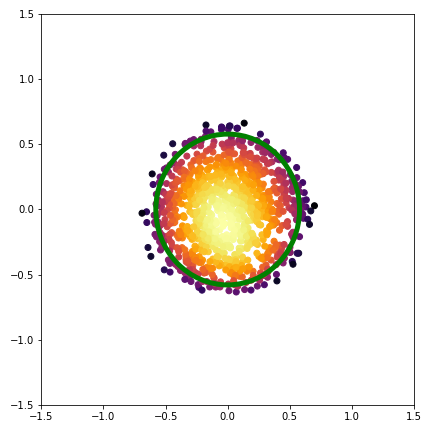}
        \caption{$t=0.01$}
    \end{subfigure}

       \caption{Computed solutions of 2D porous medium equation. The figures show the evolution of densities from $t=0$ to $t=0.12$. The algorithm calculates the positions of particles $\bx \in \mathbb{R}^2$ and their corresponding density values $\rho(t, \bx)$ at each time $t$. The color of each particle represents its density value, with brighter colors indicating higher density values. The green ring indicates the support of the density from the analytical solution at the given time $t$.}
       \label{fig:porous1}
\end{figure}

\begin{figure}[th!]
     \centering
     \begin{subfigure}[b]{0.45\textwidth}
         \centering
         \includegraphics[width=\textwidth]{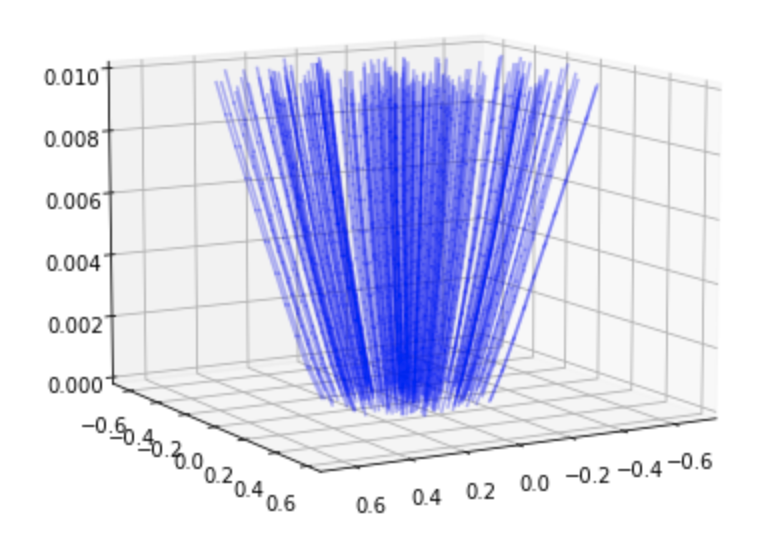}
     \end{subfigure}
        \caption{3D plot from Figure~\ref{fig:porous1}. It shows the trajectories of each particles where $z$-axis represents time from $t=0$ to $t=0.01$.}
        \label{fig:porous-3d}
\end{figure}

\begin{figure}[th!]
    \centering
    \begin{subfigure}[b]{0.19\textwidth}
        \centering
        \includegraphics[width=\textwidth]{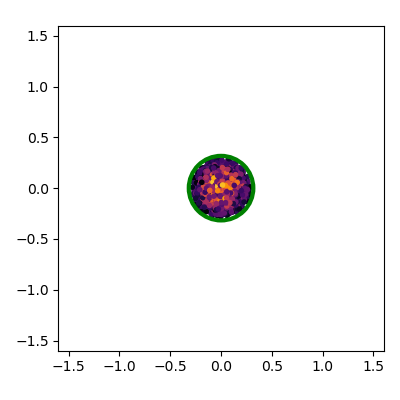}
        \caption{$t=0$}
    \end{subfigure}
    \hfill
    \begin{subfigure}[b]{0.19\textwidth}
        \centering
        \includegraphics[width=\textwidth]{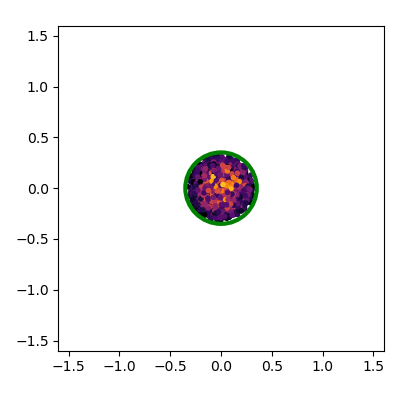}
        \caption{$t=0.0005$}
    \end{subfigure}
    \hfill
    \begin{subfigure}[b]{0.19\textwidth}
        \centering
        \includegraphics[width=\textwidth]{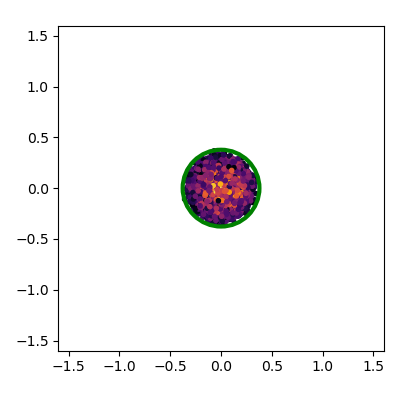}
        \caption{$t=0.001$}
    \end{subfigure}
    \begin{subfigure}[b]{0.19\textwidth}
        \centering
        \includegraphics[width=\textwidth]{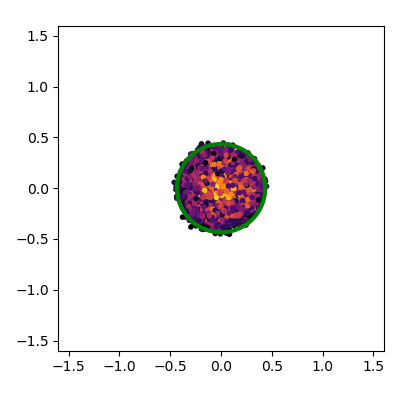}
        \caption{$t=0.0025$}
    \end{subfigure}
    \begin{subfigure}[b]{0.19\textwidth}
        \centering
        \includegraphics[width=\textwidth]{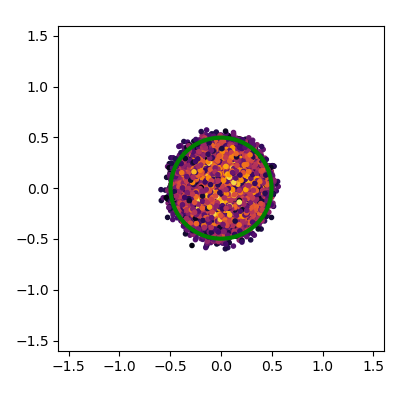}
        \caption{$t=0.005$}
    \end{subfigure}

    \begin{subfigure}[b]{0.19\textwidth}
        \centering
        \includegraphics[width=\textwidth]{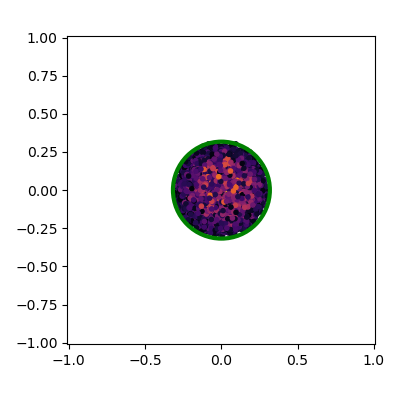}
        \caption{$t=0$}
    \end{subfigure}
    \hfill
    \begin{subfigure}[b]{0.19\textwidth}
        \centering
        \includegraphics[width=\textwidth]{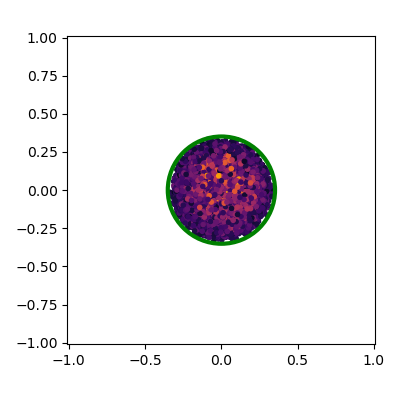}
        \caption{$t=0.0005$}
    \end{subfigure}
    \hfill
    \begin{subfigure}[b]{0.19\textwidth}
        \centering
        \includegraphics[width=\textwidth]{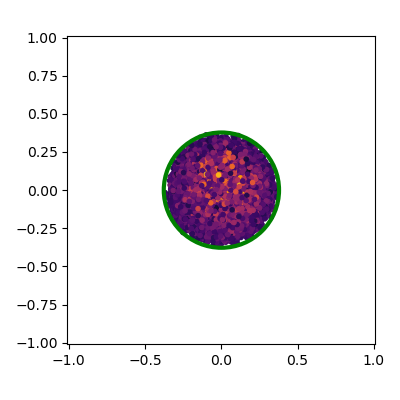}
        \caption{$t=0.001$}
    \end{subfigure}
    \begin{subfigure}[b]{0.19\textwidth}
        \centering
        \includegraphics[width=\textwidth]{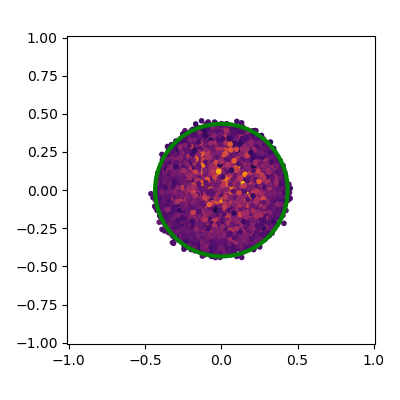}
        \caption{$t=0.0025$}
    \end{subfigure}
    \begin{subfigure}[b]{0.19\textwidth}
        \centering
        \includegraphics[width=\textwidth]{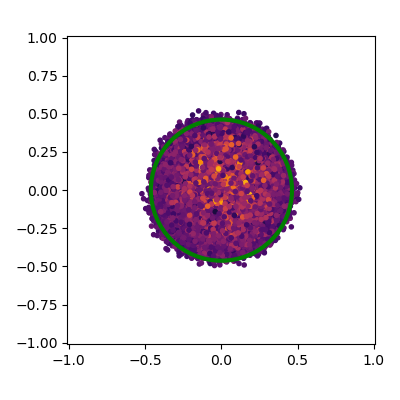}
        \caption{$t=0.0035$}
    \end{subfigure}
       \caption{Computed solutions of 6D (first row) and 10D (second row) porous medium equation. The figures show the evolution of densities from $t=0$ to $t=0.005$ for 6D and $t=0$ to $t=0.0035$ for 10D porous medium flows. The green ring indicates the support of the density from the analytical solution at the given time $t$.}
       \label{fig:porous10d}
\end{figure}




\subsection{Nonlocal transport equation with nonlinear mobility}
In this section, we explore an example that involves nonlinear mobility, a phenomenon frequently encountered in the study of transport in biological systems, where it serves to prevent overcrowding \cite{burger2006keller}. Particularly, we examine a specific mobility function, represented as $M(\rho) = \rho(1-\rho)$, and a particular energy functional described by the equation:
\begin{equation} 
\energy (\rho) = \half \int_{\RR^d \times \RR^d} W(\bx-\by) \rho(\bx) \rho(\by) \rd \bx \rd \by\,, \quad W(x) = |x|^2\,.
\end{equation}
The corresponding equation can be expressed as:
\begin{equation} \label{1105}
\partial_t\rho(t,\bx) = \nabla\cdot\Big(M(\rho(t,\bx))\nabla \frac{\delta \energy}{\delta \rho}(t,\bx)\Big)\,.
\end{equation}
This model is based on \cite{fagioli2022gradient}. Beginning with an initial condition $\rho_0(x) = \frac{3}{4} (1-\|x\|^2)_+$, it is anticipated that  $\rho_\infty$ will take the form of a characteristic function whose support is determined by the total mass. The experimental results are presented in Figure~\ref{fig:nonlinear}. The top row of figures illustrates particle evolution from a bird's-eye view, where each particle's position $\bx$ is color-coded according to its associated density value, $\rho^n(\bx)$. The bottom row of figures presents a side view of the density values for each particle when their locations are projected onto the $x$-axis. These visualizations provide compelling evidence of the density converging toward a characteristic function.

\begin{figure}[ht!]
     \centering
     \begin{subfigure}[b]{0.24\textwidth}
         \centering
         \includegraphics[width=\textwidth]{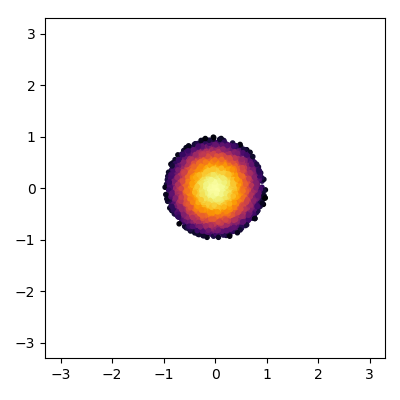}
         \caption{$t=0$}
     \end{subfigure}
     \hfill
     \begin{subfigure}[b]{0.24\textwidth}
         \centering
         \includegraphics[width=\textwidth]{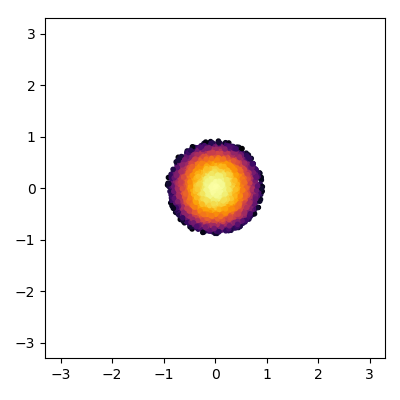}
         \caption{$t=0.02$}
     \end{subfigure}
     \hfill
     \begin{subfigure}[b]{0.24\textwidth}
         \centering
         \includegraphics[width=\textwidth]{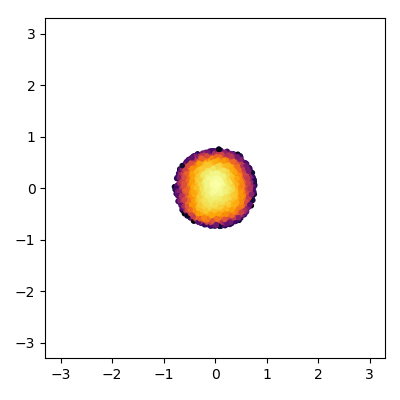}
         \caption{$t=0.1$}
     \end{subfigure}
     \begin{subfigure}[b]{0.24\textwidth}
         \centering
         \includegraphics[width=\textwidth]{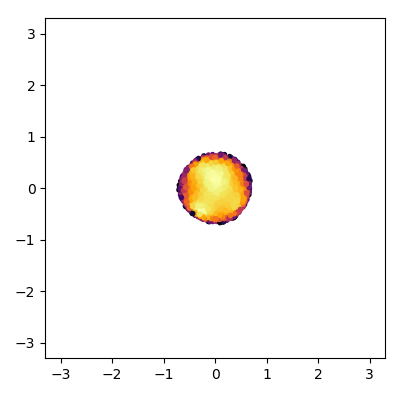}
         \caption{$t=0.2$}
     \end{subfigure}

     \begin{subfigure}[b]{0.24\textwidth}
         \centering
         \includegraphics[width=\textwidth]{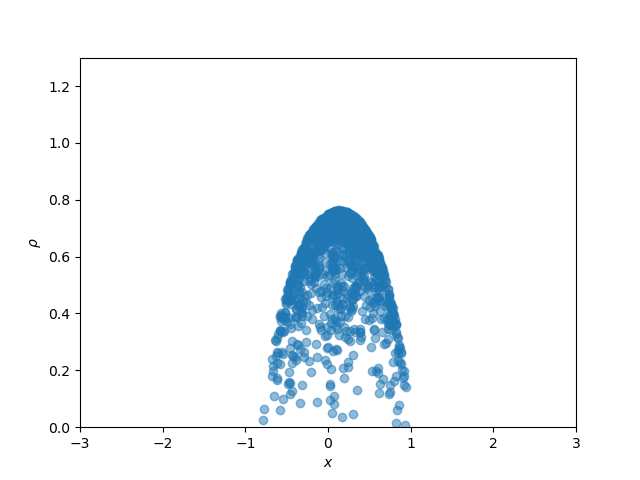}
         \caption{$t=0$}
     \end{subfigure}
     \hfill
     \begin{subfigure}[b]{0.24\textwidth}
         \centering
         \includegraphics[width=\textwidth]{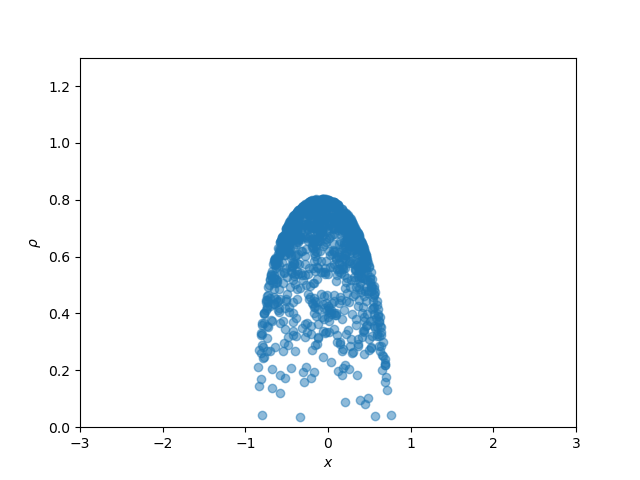}
         \caption{$t=0.02$}
     \end{subfigure}
     \hfill
     \begin{subfigure}[b]{0.24\textwidth}
         \centering
         \includegraphics[width=\textwidth]{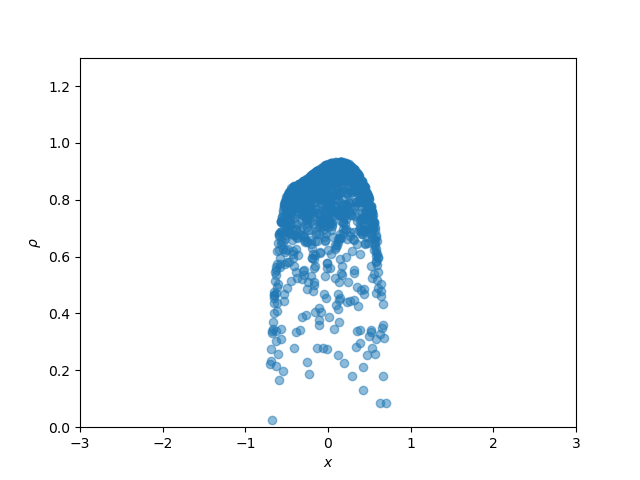}
         \caption{$t=0.1$}
     \end{subfigure}
     \begin{subfigure}[b]{0.24\textwidth}
         \centering
         \includegraphics[width=\textwidth]{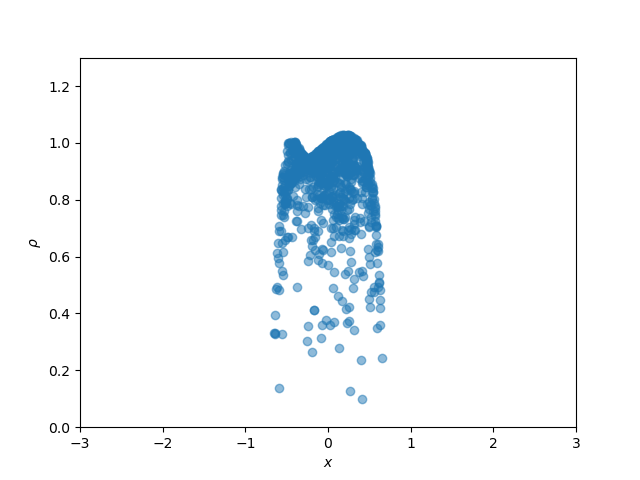}
         \caption{$t=0.2$}
     \end{subfigure}
        \caption{Computed solutions of nonlocal transport equation with nonlinear mobility \eqref{1105}.
        The figures show the evolution of densities from $t=0$ to $t=0.2$. Top row: bird eye view with color coded density. Bottom row: particle density (on the $y$-axis) versus its projected location on the $x$-axis. }
        \label{fig:nonlinear}
\end{figure}

\subsection{Kalman-Wasserstein gradient flow}
In the last example, we consider the Kalman-Wasserstein gradient flow \cite{Garbuno-InigoHoffmannLiStuart2020_interactingb}:
\begin{equation}\label{KWGF}
\partial_t\rho(t,u) = \nabla_u\cdot\Big(\rho(t,u) C(\rho)\nabla_u \frac{\delta \energy}{\delta \rho}(t,u)\Big),
\end{equation}
where
$\bu\in\mathbb{R}^d$ is the parameter in the Bayesian inverse problem, 
\begin{align}
    C(\rho) = \int_{\mathbb{R}^d} (\bu-\bm(\rho))\otimes (\bu-\bm(\rho)) \rho(\bu) \rd \bu\,,
    \quad
    \bm(\rho) = \int_{\mathbb{R}^d} \bu \rho(\bu) \rd \bu\,,
\end{align}
and 
\begin{align}\label{relative_entropy}
    \energy(\rho) = \mathrm{D}_{\mathrm{KL}}(\rho\|\pi) = \int_{\mathbb{R}^d} \Big(\Phi(\bu) \rho + \rho(\bu) \ln \rho(\bu)\Big) \rd\bu+\ln Z \,.
\end{align}
Here $\rho$ is a probability density function of $u$, $\pi=\frac{1}{Z} e^{-\Phi(\bu)}$ is a given target distribution and we assume $Z=\int e^{-\Phi(\bu)}d\bu<+\infty$ is a normalization constant. 

This equation is derived as a mean field limit of the Kalman Wasserstein sampler, aiming at identifying the unknown parameter $\bu$ in the Bayesian setting. In particular, $\Phi(\bu)$  takes the form
\begin{align} \label{Bay1}
    \Phi(\bu) = \half \norm{\mathcal G (\bu)- \by}_\Gamma^2  + \half \norm{\bu}_{\Gamma_0}^2\,,
\end{align}
where $\norm{\bu}_\Amat := \bu^\top \Amat^{-1} \bu$, 
$\mathcal G$ represents a forward map that relates the parameter $\bu$ to the measurement $\by$, typically represented as a PDE operator. We assume that $\bu$ follows a prior distribution $\mathcal N(0, \Gamma_0)$. The term $\frac{1}{2} \norm{\mathcal G (\bu)-\by}_\Gamma^2$ is commonly referred to as the likelihood function, and $\frac{1}{Z} e^{-\Phi(\bu)}$, with $Z$ being the normalizing constant, represents the posterior distribution.

Then our goal of solving \eqref{KWGF} is to sample from the posterior distribution. Instead of running a Lagenvin type dynamics as proposed in \cite{Garbuno-InigoHoffmannLiStuart2020_interactingb}, we use our DeepJKO formulation. 
Note that $C(\rho)$ dose not have an explicit $\bu$ dependence,  \eqref{non1} simplifies to 
\begin{align*} 
     & \int_{\RR^d}  \rho(\tau,\bu) \bv(\tau,\bu)^\top  
     C(\rho(\tau,\bu))^{-1}
     \bv(\tau,\bu)
     \rd \bu  \nonumber
   \\ = & \int_{\RR^d} 
\bv(\tau,\bT(\tau,\bu))^\top C (\rho(\tau, \bT(\tau,\bu)))^{-1} \bv(\tau,\bT(\tau,\bu))  \rho^n(\bu)  \rd \bu
  \\ = &  \int_{\RR^d} 
\bv(\tau,\bT(\tau,\bu))^\top C (\tau)^{-1} \bv(\tau,\bT(\tau,\bu))  \rho^n(\bu)  \rd \bu\,.
\end{align*}
It's important to highlight that $C(\tau)=C (\rho(\tau, \bT(\tau,\bu)))$ is reliant on the moments of $\rho$ and thus remains unaffected by variations in $u$.
Consequently, the JKO scheme \eqref{Mob-JKO} rewrites to:

\begin{equation}  \label{1009}
    \begin{dcases}
    \min_{\theta}  \frac{1}{N}\sum_{j=1}^N \Big\{  \int_0^1  \bv_\theta (\tau, \bT(\tau,\bu_j )^\top C(\tau)^{-1} \bv_\theta (\tau, \bT(\tau,\bu_j )) \rd \tau  
    \\ \hspace{3cm}  + 2 \Delta t \left[\Phi(\bT(1, \bu_j)) { + } \log\rho^n(\bu_j)-\log \mathrm{det}(\nabla_u \bT(1,\bu_j)) \right] \Big\}\, , 
    \\ \text{where}~~ C(\tau) = \frac{1}{N} \sum_{j=1}^N (\bT(\tau, \bu_j) -\bm ) \otimes (\bT(\tau, \bu_j) -\bm) \,,
    \quad
    \bm = \frac{1}{N} \sum_{j=1}^N \bT(\tau, \bu_j)\, ,
    \\ \text{s.t.} ~~~\frac{\rd}{\rd \tau} \bT(\tau, \bu_j) = \bv_\theta  (\tau, \bT(\tau,\bu_j))\, , \qquad \bT(0,\bu_j) = \bu_j\,. 
   \\ \qquad     \frac{\partial}{\partial \tau} \log \det (\nabla_{\bu} \bT(\tau, \bu_j))= \divergence (\bv) (\tau, \bT(\tau,\bu_j))\,, \quad \log \det (\nabla_{\bu} \bT(0, \bu_j))=0\,.
    \end{dcases}
\end{equation}

As a specific example, we consider a case where the forward map $\mathcal G$ is given by the one-dimensional elliptic boundary value problem:
\begin{align*}
    -\frac{\rd }{\rd x} \left( e^{u_1} \frac{\rd }{\rd x} p_{\bu}(x) \right) = 1, \quad  x \in [0,1] \,,
\end{align*}
with boundary conditions $p_{\bu}(0) = 0$ and $p_{\bu}(1) = u_2$. This problem has been considered in \cite{herty2019kinetic, Garbuno-InigoHoffmannLiStuart2020_interactingb, li2023differential}. The explicit solution for this problem is 
\begin{align*}
    p_{\bu}(x) = u_2 x  + e^{-u_1} \left( -\frac{x^2}{2} + \frac{x}{2}\right)\,. 
\end{align*}
Then the forward map $\mathcal G$ is defined by 
\begin{align}
    \mathcal G(\bu) = (p_{\bu}(0.25), p_{\bu}(0.75))^\top\,, \quad \text{where}~ \bu = (u_1, u_2)^\top\,.
\end{align}
Then the Bayesian inverse problem is to find the distribution of the unknown $\bu$ conditioned on the observation $\by$, assuming additive Gaussian noise. More precisely, we use $\Phi(\bu)$ defined in \eqref{Bay1} and choose 
\begin{align*}
  \by = (27.5, 79.7)^\top, \quad  \Gamma = 0.1^2 \Imat_2, \quad \Gamma_0 = 10^2 \Imat_2, 
\end{align*}
where $\Imat_2 \in \RR^{2\times 2}$ is the identity matrix. We run the dynamics by initially sampling according to 
\begin{align}
    \rho_0(\bu) = \mathcal N (0,1) \times \mathcal U(90,110)\,,
\end{align}
That is, sample $u_1$ according to normal distribution $\mathcal N (0,1) $, and $u_2$ according to uniform distribution $\mathcal U(90,110)$. 

In the computational approach, we set the matrix $C(\tau)^{-1}$ to be equal to $C(0)^{-1}$ to approximate the running costs. This approximation offers the advantage of reducing computational complexity by circumventing the need for backpropagation over the inverse cost function, thereby significantly increasing computation speed. Note that the approximation introduces an error of order $\Delta t$, which does not compromise the first order accuracy of the JKO scheme. The results of our experiment are presented in Figure~\ref{fig:kalman}, illustrating the density evolution from $t=0$ to $t=20$. The density converges to the stationary distribution, and the results are consistent with the numerical solution depicted in~\cite{Garbuno-InigoHoffmannLiStuart2020_interactingb}. 

Figure~\ref{fig:kalman-conv} shows the loss plot and the contour plot of $\Phi(\bu)$. The loss plot displays the value of the loss function with the x-axis representing iterations and the y-axis representing the loss value. Here the loss function represents the relative entropy defined in \eqref{relative_entropy}. 
On the right, the figure illustrates the computed solution at $t=20$, overlaid on the contour plot derived from $\Phi(\bu)$.
        \begin{figure}[th!]
            \centering
            \begin{subfigure}[b]{0.19\textwidth}
                \centering
                \includegraphics[width=\textwidth]{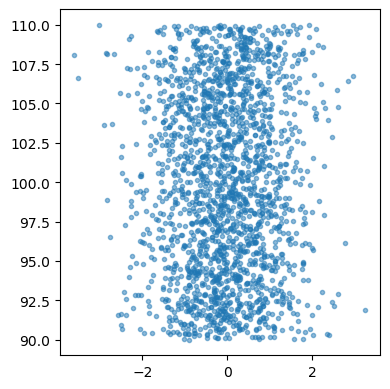}
                \caption{$t=0$}
            \end{subfigure}
            \hfill
            \begin{subfigure}[b]{0.19\textwidth}
                \centering
                \includegraphics[width=\textwidth]{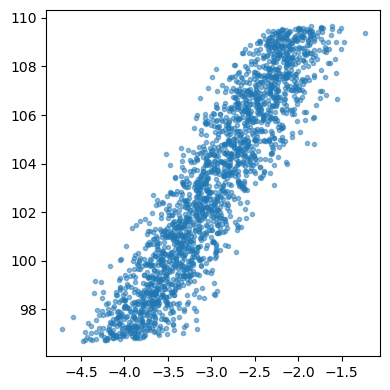}
                \caption{$t=2$}
            \end{subfigure}
            \hfill
            \begin{subfigure}[b]{0.19\textwidth}
                \centering
                \includegraphics[width=\textwidth]{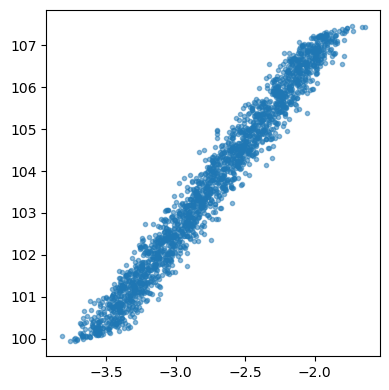}
                \caption{$t=4$}
            \end{subfigure}
            \begin{subfigure}[b]{0.19\textwidth}
                \centering
                \includegraphics[width=\textwidth]{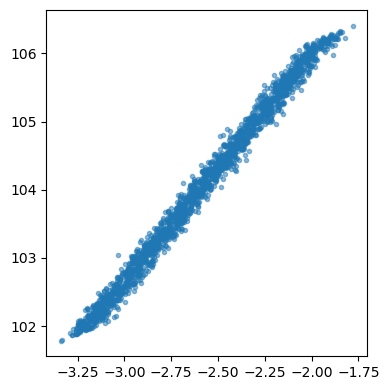}
                \caption{$t=7.5$}
            \end{subfigure}
            \begin{subfigure}[b]{0.19\textwidth}
                \centering
                \includegraphics[width=\textwidth]{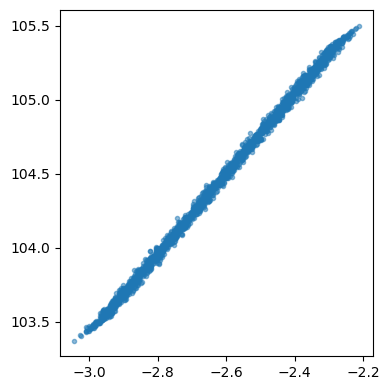}
                \caption{$t=20$}
            \end{subfigure}
               \caption{Computed solutions of a 2D Kalman-Wasserstein gradient flow \eqref{KWGF}. }
               \label{fig:kalman}
        \end{figure}
\begin{figure}[th!]
    \centering
    \begin{subfigure}[b]{0.5\textwidth}
        \centering
        \includegraphics[width=\textwidth]{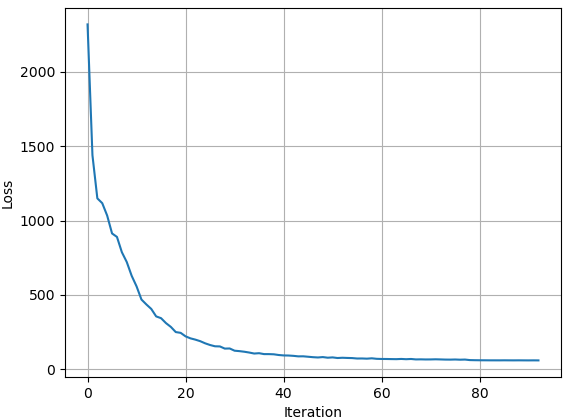}
        \caption{Loss plot with respect to outer iterations}
    \end{subfigure}
    \hfill
    \begin{subfigure}[b]{0.4\textwidth}
        \centering
        \includegraphics[width=\textwidth]{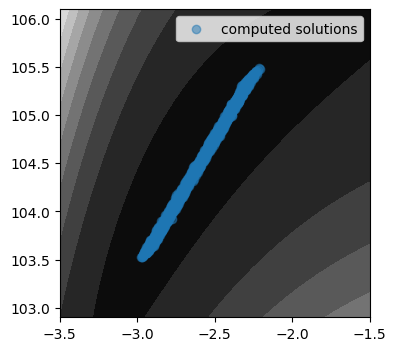}
        \caption{Computed solution and a contour plot of $\Phi(\bu)$}
    \end{subfigure}
    \hfill
       \caption{The loss function values across iterations are depicted on the x-axis, while the computed solution at $t=20$ is superimposed on the contour plot generated from $\Phi(\bu)$ (on the right).}
       \label{fig:kalman-conv}
\end{figure}

\section{Conclusion and discussion}\label{sec5}
In this paper, we introduce neural network-based implicit particle methods for computing high-dimensional Wasserstein-type gradient flows with linear and nonlinear mobility functions. Our approach centers around the Lagrangian formulation within the Jordan–Kinderlehrer–Otto (JKO) framework, where neural networks approximate the velocity field. Neural ODE techniques are harnessed to efficiently compute the time-implicit updates of both particle locations and the densities they carry. Additionally, we leverage an explicit recurrence relation to compute derivatives, significantly streamlining the backpropagation process. Our methodology exhibits versatility and can handle a broad spectrum of gradient flows while accommodating diverse potential functions and nonlinear mobility scenarios.

 There are several promising directions for future work. Firstly, there is potential for accelerating the training process. This could involve leveraging the properties of the velocity field as suggested in \cite{liu2022flow}, or exploring transfer learning techniques \cite{xu2023transfer}, and improving resampling methods.
Another avenue for exploration is the convexity of the problem, investigating how it depends on parameters like the step size $\Delta t$ and network architectures, including the choices of neural network activation functions. Understanding the convexity properties can contribute to faster training and provide insights into the convergence properties of the algorithm.
Additionally, we aim to extend this approach to accommodate more complex energy functionals and scenarios where a gradient flow structure is not fully presented or is only part of the physical dynamics. This extension would broaden the applicability of our method to a wider range of problems in simulating general physics oriented partial differential equations.
\bibliographystyle{siam} 
\bibliography{DeepJKO_ref.bib}
\end{document}